\documentclass[11pt]{amsart}

\usepackage{latexsym}
\usepackage{amssymb}
\usepackage{amsmath}
\usepackage{fancybox,color}
\usepackage{amssymb}
\usepackage{amsfonts}
\usepackage{latexsym}
\usepackage{amsmath,systeme}
\usepackage[leqno]{mathtools}

\usepackage{amsmath}
\usepackage{amsthm}
\usepackage{color}
\usepackage{amssymb}
\usepackage{graphicx}
\usepackage{mathrsfs}
\usepackage{multicol}
\usepackage{wrapfig}
\usepackage{times}
\usepackage{enumerate}
\newtheorem{definition}{Definition}
\newtheorem{remark}{Remark}
\newtheorem{Lemma}{Lemma}
\newtheorem{theorem}{Theorem}
\newtheorem{proposition}{Proposition}
\newtheorem{corol}{Corollary}

\usepackage{mathtools}
\mathtoolsset{showonlyrefs}

\def\vint_#1{\mathchoice%
          {\mathop{\kern 0.2em\vrule width 0.6em height 0.69678ex depth -0.58065ex
                  \kern -0.8em \intop}\nolimits_{\kern -0.4em#1}}%
          {\mathop{\kern 0.1em\vrule width 0.5em height 0.69678ex depth -0.60387ex
                  \kern -0.6em \intop}\nolimits_{#1}}%
          {\mathop{\kern 0.1em\vrule width 0.5em height 0.69678ex
              depth -0.60387ex
                  \kern -0.6em \intop}\nolimits_{#1}}%
          {\mathop{\kern 0.1em\vrule width 0.5em height 0.69678ex depth -0.60387ex
                  \kern -0.6em \intop}\nolimits_{#1}}}
\def\vintslides_#1{\mathchoice%
          {\mathop{\kern 0.1em\vrule width 0.5em height 0.697ex depth -0.581ex
                  \kern -0.6em \intop}\nolimits_{\kern -0.4em#1}}%
          {\mathop{\kern 0.1em\vrule width 0.3em height 0.697ex depth -0.604ex
                  \kern -0.4em \intop}\nolimits_{#1}}%
          {\mathop{\kern 0.1em\vrule width 0.3em height 0.697ex depth -0.604ex
                  \kern -0.4em \intop}\nolimits_{#1}}%
          {\mathop{\kern 0.1em\vrule width 0.3em height 0.697ex depth -0.604ex
                  \kern -0.4em \intop}\nolimits_{#1}}}

\newcommand{\kint}{\vint}

\newcommand{\aveint}[2]{\mathchoice%
          {\mathop{\kern 0.2em\vrule width 0.6em height 0.69678ex depth -0.58065ex
                  \kern -0.8em \intop}\nolimits_{\kern -0.45em#1}^{#2}}%
          {\mathop{\kern 0.1em\vrule width 0.5em height 0.69678ex depth -0.60387ex
                  \kern -0.6em \intop}\nolimits_{#1}^{#2}}%
          {\mathop{\kern 0.1em\vrule width 0.5em height 0.69678ex depth -0.60387ex
                  \kern -0.6em \intop}\nolimits_{#1}^{#2}}%
          {\mathop{\kern 0.1em\vrule width 0.5em height 0.69678ex depth -0.60387ex
                  \kern -0.6em \intop}\nolimits_{#1}^{#2}}}
\newcommand{\ol}{\overline}
\newcommand{\Om}{\Omega}
\newcommand{\R}{\mathbb{R}}
\newcommand{\E}{\mathbb{E}}

\newcommand{\W}{\mathbb{W}}
\renewcommand{\P}{\mathbb{P}}

\newcommand{\eps}{\varepsilon}

\newcommand{\half}{\frac{1}{2}}
\newcommand{\I}{\textrm{I}}
\newcommand{\II}{\textrm{II}}

\begin{document}
\large
\title{A game theoretical approximation for a parabolic/elliptic system with different operators}
\author{Alfredo Miranda and Julio D. Rossi}

\address{Alfredo Miranda and Julio D. Rossi
\hfill\break\indent
Departamento  de Matem{\'a}tica, FCEyN,
Universidad de Buenos Aires,
\hfill\break\indent Pabellon I, Ciudad Universitaria (1428),
Buenos Aires, Argentina.}

\email{{\tt amiranda@dm.uba.ar, jrossi@dm.uba.ar}}

\date{}

\maketitle

\centerline{\it To Juan Luis Vazquez in his 75th anniversary with our best wishes.}

\begin{abstract}
In this paper we find viscosity solutions to a coupled system composed by two equations,
the first one is parabolic and driven by the infinity Laplacian while the second one is elliptic and involves the usual
Laplacian. We prove that there is a two-player zero-sum game played in two different boards with different rules in each board 
(in the first one we play a Tug-of-War game
taking the number of plays into consideration and in the second board we move at random) 
whose value functions converge uniformly to a viscosity solution
to the PDE system.
\end{abstract}

\section{Introduction}

Our main goal in this paper is to provide a probabilistic approach to find solutions to
 an elliptic/parabolic system in which we have two different operators. We deal with viscosity solutions to 
\begin{equation}
\label{ED1}
 \left\lbrace
\begin{array}{ll}
\displaystyle \frac{\partial u}{\partial t} (x,t) - \displaystyle \half \Delta_{\infty}u(x,t) + u(x,t) - v(x,t)=0 \qquad &  \ x \in \Omega,\, t>0,  \\[10pt]
-   \displaystyle \frac{\kappa}{2} \Delta v(x,t) + v(x,t) - u(x,t)=0  \qquad &  \ x \in \Omega, \, t>0, \\[10pt]
u(x,t) = f(x,t) \qquad & \ x \in \partial \Omega, \, t > 0,  \\[10pt]
v(x,t) = g(x,t) \qquad & \ x \in \partial \Omega, \, t > 0, \\[10pt]
u(x,0) = u_0 (x) \qquad &\  x \in \Omega.
\end{array}
\right.
\end{equation}

Notice that this system involves two differential operators, the usual Laplacian 
$$ \Delta \phi = \sum\limits_{i=1}^{N} \partial_{x_{i}x_{i}}\phi$$ and the infinity Laplacian (see \cite{Cran})
 $$
 \Delta_{\infty}\phi =  
 \langle D^{2}\phi \frac{\nabla \phi}{|\nabla \phi |},\frac{\nabla \phi}{| \nabla \phi |}\rangle 
 =\frac{1}{| \nabla \phi |^2}\sum\limits_{i,j=1}^{N} \partial_{x_{i}}\phi \partial_{x_{i}x_{j}}\phi
 \partial_{x_{j}}\phi.$$ 
In addition, the first equation is parabolic (it involves a time derivative of $u$) while the second one
is elliptic (the time variable can be viewed just as a parameter in the equation for $v$).
Remark that there is an initial condition for $u$ but not for $v$. This is due to the fact that
the first equation is parabolic while the second one is elliptic. 

There is a large literature highlighting the interplay between Probability and Partial Differential 
Equations. In fact, there is a deep connection between classical potential theory and probability theory.
The main idea is that harmonic functions 
and martingales have something in common: the validity of mean value formulas.
A well known fact 
is that $u$ is harmonic, that is $u$ verifies the PDE $\Delta u =0$, 
if and only if it verifies the mean value property
$
u(x) =
\frac{1}{|B_\varepsilon (x) |} \int_{B_\varepsilon (x) } u(y) \, dy.
$
In fact, we can relax
this condition by requiring that it holds asymptotically
$
u(x) =
\frac{1}{|B_\varepsilon (x) |} \int_{B_\varepsilon (x) } u(y) \, dy + o(\varepsilon^2),
$
as $\varepsilon\to 0$.
The connection between the Laplacian and the Bownian motion or with the limit
of random walks as the step size goes to zero is also well known, see \cite{Kac}. 

Nowadays, it is known that the ideas and techniques used for linear equations 
can be extended to cover nonlinear problems.
For a mean value property for the $p-$Laplacian (including 
the infinity Laplacian) we refer to \cite{I,dTLin,KMP,LM} and \cite{MPR}. 
See also \cite{BChMR} for mean value formulas for Monge-Ampere. 
These mean value formulas are closely related to game theoretical approximations of solutions
to PDEs. 
For a probabilistic
approximation of the infinity Laplacian
there is a game (called Tug-of-War game in the literature) that was introduced in \cite{PSSW} and 
generalized in several directions to cover other equations, like the $p-$Laplacian, in \cite{TPSS,AS,BR,ChGAR,LPS,MPR,
MPRa,MPRb,Mitake,PS,R}. There are also parabolic versions of these results, we refer to \cite{BER,MPRparab}. 
For a general overview of the subject we refer to the recent
books \cite{BRLibro}  and \cite{Lewicka} and references therein.

For systems of equations there are less references available. This is due to the fact that for general
fully nonlinear systems there is no viscosity theory and also that the estimates needed to pass to
the limit in the approximations are more involved (usually for systems one needs to obtain estimates for both
components simultaneously). For elliptic systems we quote \cite{Mitake} and \cite{nosotros}. 
In \cite{Mitake} a coupled elliptic system involving the infinity Laplacian for every component was analyzed; while in
\cite{nosotros} an elliptic system involving two different operators (the Laplacian and the infinity Laplacian) was studied.
This paper can be viewed as a follow up of these two references. The main difference with 
\cite{Mitake} and \cite{nosotros} is that here we have to tackle a time dependent problem and then
we have to take extra care when we obtain estimates for the components of the approximations
(we need to prove estimates in space and time simultaneously).

 The system \eqref{ED1} is not variational (there is no associated energy). Therefore, to find solutions
  one possibility is to use monotonicity methods based in a comparison principle (Perron's argument). Here we will look at the system
  in a different way and to obtain existence of solutions we find an approximation using game theory.
 This approach not only gives existence of solutions but it also provide us with a description that
 yields some intuition on the behaviour of the solutions.  
 At this point we note that we will understand solutions to the system in a viscosity 
 sense. This is natural since the infinity Laplacian is not variational (see Section \ref{sect-prelim} for
 the precise definition). Once we have existence of solutions, we prove a comparison principle that implies uniqueness. 
 
Now let us describe the game that is associated with \eqref{ED1}. Fix $T>0$ (we aim to obtain a continuous viscosity solution to \eqref{ED1}
in $\overline{\Omega}\times [0,T]$ with $T$ arbitrary). 
The game is a two-player zero-sum game played in two different boards. We will call board to a cylinder of the form $\Om\times (0,T]\subseteq \R^N \times[0,T]$. 
Fix a parameter, $\eps>0$ and two final payoff functions 
$ \ol{f},  \ol{g} : \mathbb{R}^N \setminus \Omega\ \times[0,T] \mapsto \mathbb{R}$ (one for each board $ \ol{f}$ for the first board and $ \ol{g}$ for the second one). These payoff functions  $ \ol{f}$ and $ \ol{g}$ are just two Lipschitz extensions to $\{\mathbb{R}^N \setminus \Omega\}\times[0,T]$ of the boundary data $f$ and $g$ that appear in \eqref{ED1}. Fix also a final payoff $u_0: \Omega \mapsto \mathbb{R}$ 
to be used in the first board if we reach a nonpositive time. If we are playing at a point $(x,t)\in\Om\times(0,T]$ in the first board, 
with probability $1-\eps^2$ the players play 
the Tug of War game (a fair coin is tossed and the winner choses the new position of the game in
the ball $B_{\eps}(x)$), but descending to the level $t-\eps^2$, that is, the next position of the game will be a point that looks like $(y,t-\eps^2)$ with $y\in B_{\eps}(x)$ that depends on the choice of the players and the toss of a fair coin (the winner of the coin toss chooses the next position 
in $B_{\eps}(x)$). On the other hand, playing at a point $(x,t)\in\Om\times(0,T]$ in the second board, with probability $1-\eps^2$ the next position is random in $B_{\eps}(x)$ at the same time level. 
That is, the next point looks like $(y,t)$, where $y\in B_{\eps}(x)$ is chosen at random  (with uniform distribution). On top of these rules, being in either of the two boards we can jump to the other board with probability $\eps^2$. That is, if we are at $(x,t)\in\Om\times(0,T]$ on one of the boards, with probability $\eps^2$ we will go to $(x,t)$ but on the other board. In some cases we will add an extra index and denote a point by $(x,t,j)$ where $j=1,2$ indicates the board. In this context, if we are at $(x,t,1)$, with probability $1-\eps^2$ we will go to $(y,t-\eps^2,1)$ with $y\in B_{\eps}(x)$ is chosen playing Tug of War, and with probability $\eps^2$ we will go to $(x,t,2)$. We can think in a similar way if we are at $(x,t,2)$, with probability $1-\eps^2$ we will go to $(y,t,1)$ with $y\in B_{\eps}(x)$ chosen at random, and with probability $\eps^2$ we will go back to the first board at $(x,t,1)$. The game continues until the position of the token leaves the domain, or the time is below zero (this can only happen playing in the first board) and at this final
point $(x_\tau,t_\tau,j)$ with $t_\tau>0$ the first player gets $\ol{f}(x_\tau, t_\tau)$ and the second player $- \ol{f}(x_\tau, t_\tau)$ if they are
playing in the first board (that is, if $j=1$) while they obtain $\ol{g}(x_\tau, t_\tau)$ and $- \ol{g}(x_\tau, t_\tau)$ if they are
playing in the second board ($j=2$), or when $t_\tau \leq 0$ the first player gets $u_0(x_\tau)$ and the second player $-u_0(x_\tau)$.
(we can think that Player $\II$ pays to Player $\I$ the amount prescribed by $\ol{f}$, $\ol{g}$ or $u_0$, according to the final position of the game).

This game has a expected value (the best outcome of the game that both players expect to obtain
playing their best,
see Section \ref{sect-Game} for a precise definition). In this case the value of the game
is given by a pair of functions $(u^\eps, v^\eps)$, defined in $\Omega\times(0,T]$ that depends on the size of the steps, $\eps$. 
For $(x_0,t_0) \in \Omega\times(0,T]$, the value of $u^\eps (x_0,t_0)$
is the expected outcome of the game when it starts at $(x_0,t_0)$ in the first board, while $v^\eps (x_0,t_0)$ is the expected value 
starting at $(x_0,t_0)$ in the second board.

Our first theorem ensures that this game has a well-defined value and that this pair of functions $(u^\eps, v^\eps)$
 verifies a system of equations (called the dynamic programming principle (DPP)
in the literature).

\begin{theorem} \label{teo.dpp2}
The game has value 
$
(u^\eps, v^\eps)
$
that is the unique solution to  
\begin{equation}
\label{DPP}
 \left\lbrace
\begin{array}{ll}
 \displaystyle u^{\eps}(x,t)=\eps^{2}v^{\eps}(x,t)+(1-\eps^{2})\Big\{\half \sup_{y \in B_{\eps}(x)}u^{\eps}(y,t-\eps^{2}) + \half \inf_{y \in B_{\eps}(x)}u^{\eps}(y,t-\eps^{2})\Big\}, 
  &  x \in \Omega, t\in(0,T],  \\[10pt]
   \displaystyle v^{\eps}(x,t)=\eps^{2}u^{\eps}(x,t)+(1-\eps^{2})\kint_{B_{\eps}(x)}v^{\eps}(y,t)dy,   &   x \in \Omega,t\in(0,T],  \\[10pt]
u^{\eps}(x,t) = \ol{f}(x,t) \qquad & x \in \R^{N} \backslash \Omega, t > 0,  \\[10pt]
v^{\eps}(x,t) = \ol{g}(x,t) \qquad &  x \in \R^{N} \backslash \Omega, t > 0,\\[10pt]
u^{\eps}(x,t)=u_0 (x) \qquad &  x \in \Omega, t\leq 0.
\end{array}
\right.
\end{equation}
\end{theorem}

Notice that \eqref{DPP} can be seen as a sort of mean value property (at size $\eps$) for the system \eqref{ED1}. 
Let see intuitively why the DPP \eqref{DPP} holds. Playing in the first board,
at each step Player $1$ chooses the next position of the game with probability
$\frac{1-\eps^2}{2}$ and aims to obtain $\sup_{y \in B_{\eps}(x)}u^{\eps}(y,t-\eps^2)$
 (recall this player seeks to maximize the expected payoff and that time decreases by $\eps^2$ each time we play in the first board); 
 with probability
$\frac{1-\eps^2}{2}$ it is player $2$ who chooses and aims to obtain $\inf_{y \in B_{\eps}(x)}u^{\eps}(y,t-\eps^2)$, and finally
with probability $\eps^2$ the game changes boards keeping the same position and time (and therefore $v^\eps (x,t)$ comes into play). 
Playing in the second board, with probability $1-\eps^2$ the point moves at random (but stays in the second board and keeps the same time) 
and hence the term $\kint_{B_{\eps}(x)}v^{\eps}(y,t)dy $ appears, but with probability $\eps^2$ the board is changed to the first one and 
hence we have $u^\eps (x,t)$ in the second equation.

Our next goal is to look for the limit as $\eps \to 0$. Our main result in this paper is to show that,
under appropriate regularity conditions on the data, $\partial \Omega$, $f$ and $g$,
these value functions $u^\eps, v^\eps$ converge uniformly in $\overline{\Omega}$ to continuous limits $u,v$ that are characterized as being a viscosity solution to \eqref{ED1}.

\begin{theorem} \label{teo.converge2} Assume that $\Omega$ is a bounded domain satisfying a 
uniform exterior sphere condition and that the data $f$, $g$ and $u_0$ are Lipschitz continuos and that the compatibility
condition $u_0(x) = f(x,0)$ for $x\in \partial \Omega$ holds.
Let $(u^\eps, v^\eps)$ denote the values of the game.
Then, there exists a pair of continuous functions in $\overline{\Omega}\times [0,T]$, $(u,v)$, such that
\begin{equation*}
u^\eps \to u, \quad \mbox{and} \quad v^\eps \to v, \qquad \mbox{ as } \eps \to 0,
\end{equation*}
uniformly in $\overline{\Omega}\times [0,T]$.
Moreover, the limit $(u,v)$ is characterized as the unique viscosity
solution to \eqref{ED1}  (with the constant $\kappa=\frac{1}{\lvert B_{1}(0)\rvert}\int_{B_{1}(0)}z_{j}^{2}dz,
$ that depends only on the dimension).
\end{theorem}

\begin{remark}{\rm It is enough to ask for a uniform modulus of continuity of the data
$f$, $g$ and $u_0$ (keeping the compatibility
$u_0(x) = f(x,0)$ for $x\in \partial \Omega$). We prefer to state and prove
our results for Lipschitz continuos functions to slightly simplify 
some of the arguments.}
\end{remark}

\begin{remark} {\rm
If we assume that the probability of moving random in the second board is $1-K\eps^2$ (and
hence the probability of changing from the second to the first board is $K\eps^2$) with the same computations we obtain 
$$
v^{\eps}(x,t)=K \eps^{2}u^{\eps}(x,t)+(1- K \eps^{2})\kint_{B_{\eps}(x)}v^{\eps}(y,t)dy
$$
 as the second equation in the DPP (the first equation and the exterior and initial data remain unchanged). Passing to the limit we get
 $$
 -   \displaystyle \frac{\kappa}{2 K } \Delta v(x,t) + v(x,t) - u(x,t)=0 ,
 $$
 and hence, choosing $K$, we can obtain any positive constant in front of the Laplacian in \eqref{ED1}. 
}
\end{remark}

Let us comment briefly on the ideas used in the proofs. To prove that the sequence converges we will apply an Arzel\`{a}-Ascoli type lemma (see Lemma \ref{lem.ascoli.arzela} in Section
 \ref{sect-uniform}).
To this end we need to prove a uniform bound and a sort of asymptotic continuity that is based on estimates
for both components $(u^\eps, v^\eps)$ near the boundary (these estimates can be extended
to the interior via a coupling probabilistic argument). 
In fact, to see an asymptotic continuity close to a boundary point, we are able to show that both players have strategies that 
enforce the game to end near a point $y\in \partial \Omega$ with 
high probability if we start close to that point no mater the strategy choosed by the other player.  
Also we need to show that the players have strategies that force the game to end close to $(y,0)$, $y \in \Omega$, when the starting point $(x,t)$
has $x$ close to $y$ and $t>0$ but small. 
This allows us to obtain a sort of asymptotic equicontinuity close to the boundary leading to
uniform convergence in the whole $\overline{\Omega}\times [0,T]$. Note that, in general the value functions
$(u^\eps,v^\eps)$ are discontinuous inside $\Omega\times (0,T)$ (this is due to the fact that we make discrete steps)
and therefore showing uniform convergence to continuos limits is a difficult task.
Once we prove uniform convergence along subsequences of $(u^\eps, v^\eps)$ to a continuous limit $(u, v)$
we show that the limit is in fact a viscosity solution to the system \eqref{ED1} (here we use viscosity arguments 
taking into account that one equation in \eqref{ED1} is parabolic and the other one elliptic).

Our elliptic/parabolic system \eqref{ED1} has a comparison principle 
(for viscosity super and subsolutions).
Therefore, we have uniqueness of the limit and we conclude convergence of the whole
family $(u^\eps, v^\eps)$ (the Arzel\`{a}-Ascoli type lemma gives us convergence along subsequences).

Finally, we observe that similar ideas could be applied to the following system,
\begin{equation}
\label{ED1.22}
 \left\lbrace
\begin{array}{ll}
\displaystyle - \displaystyle \half \Delta_{\infty}u(x,t) + u(x,t) - v(x,t)=0 \qquad &  \ x \in \Omega,\ t>0,  \\[10pt] 
  \displaystyle  \frac{\partial v}{\partial t} (x,t) - \frac{\kappa}{2} \Delta v(x,t) + v(x,t) - u(x,t)=0  \qquad &  \ x  \in \Omega, \ t>0, \\[10pt]
u(x,t) = f(x,t) \qquad & \ x \in \partial \Omega, \ t>0, \\[10pt]
v(x,t) = g(x,t) \qquad &  \ x \in \partial \Omega, \ t>0, \\[10pt]
v(x,0) = v_0 (x) \qquad &\  x \in \Omega,
\end{array}
\right.
\end{equation}
in which the time derivative appears in the second equation. 

\medskip

The paper is organized as follows: in Section \ref{sect-prelim} we gather some preliminary results
(including the precise definition of a viscosity solution to our parabolic/elliptic system); in Section 
\ref{sect-Game} we describe in detail the game; in Section \ref{sect-DPP} we start the analysis of the game and
include a proof of the existence and uniqueness of solutions to the DPP (proving Theorem \ref{teo.dpp2}) and 
in Section \ref{sect-uniform} we show the uniform convergence of the values of the game to a continuous limit;
in Section \ref{sect-limite.viscoso} we prove that this uniform limit is in fact a viscosity solution to our system
and we include a brief sketch of the proof of the comparison principle for \eqref{ED1} that implies uniqueness
of the limit.

\section{Preliminaries.} \label{sect-prelim}

In this section we first include the precise definition of what we understand as a viscosity 
solution for the system \eqref{ED1}. In this case we must consider two different definitions, one for each equation. 
Next, we include the precise statement of the \textit{Optional Stopping Theorem} that we will need when dealing with the probabilistic 
part of our arguments. 

\subsection{Viscosity solutions} We refer to
\cite{CIL} for general results on viscosity solutions.

For the first equation of \eqref{ED1} we introduce the following definition of being a
viscosity solution to a parabolic PDE. Fix a function
\[
P:\Omega\times(0,T]\times\R\times\R\times\R^N\times\mathbb{S}^N\to\R
\]
where $\mathbb{S}^N$ denotes the set of symmetric $N\times N$ matrices.We want to consider the PDE 
\begin{equation}
\label{eqvissol}
P\Big(x,t,u (x,t),\frac{\partial u}{\partial t}(x,t), Du (x,t), D^2u (x,t) \Big) =0, \qquad x \in \Omega, \ t\in(0,T).
\end{equation}
In our system we have
\begin{equation}
\label{eqvissol.888}
P\Big(x,t,r,s,p, M \Big) = s - \displaystyle \half \langle M \frac{p}{|p|} ,\frac{p}{|p|} \rangle + r - v(x,t)=0.
\end{equation}
The idea behind Viscosity Solutions is to use the maximum principle in order to 
``pass derivatives to smooth test functions''. This idea allows us to consider operators in non divergence form.
We will assume that $P$ satisfies two monotonicity properties,
\[
X\leq Y \text{ in } \mathbb{S}^N \implies P(x,t,r,s,p,X)\geq P(x,t,r,s,p,Y)
\]
for all $(x,t,r,s,p)\in \Omega\times(0,T]\times\R\times\R\times\R^N$; and
\[
s_1\leq s_2 \text{ in } \mathbb{R} \implies P(x,t,r,s_1,p,X)\leq P(x,t,r,s_2,p,Y)
\]
for all $(x,t,r,p,X)\in \Omega\times(0,T]\times\R\times\R^N \times \mathbb{S}^N$.
Here we have an equation that involves the $\infty$-laplacian that are not defined when the gradient vanishes. 
In order to be able to handle this issue, we need to consider the lower semicontinous, $P_*$, and upper semicontinous, $P^*$, envelopes of 
$P$.
These functions are given by
\[
\begin{array}{ll}
P^*(x,t,r,s,p,X)& \displaystyle =\limsup_{(y,l,m,n,q,Y)\to (x,t,r,s,p,X)}P(y,l,m,n,q,Y),\\
P_*(x,t,r,s,p,X)& \displaystyle =\liminf_{(y,l,m,n,q,Y)\to (x,t,r,s,p,X)}P(y,l,m,n,q,Y).
\end{array}
\]
These functions coincide with $P$ at every point of continuity of $P$ and are lower and upper semicontinous respectively.
With these concepts at hand we are ready to state the definition of a viscosity solution to
\eqref{eqvissol}.

\begin{definition}
\label{def.sol.viscosa.1}
An upper semi-continuous function $u$ is a subsolution of \eqref{eqvissol} if
for every $ \phi \in C^2(\Om\times(0,T])$ such that $\phi$ touches $u$ at $ (x,t) \in
\Omega\times(0,T]$ strictly from above (that is, $u-\phi$ has a strict maximum at $(x,t)$ with $u(x,t) = \phi(x,t)$), we have
$$P_*\Big(x,t,\phi(x,t),\frac{\partial\phi}{\partial t}(x,t), D \phi(x,t),D^2\phi(x,t) \Big)\leq 0.$$

A lower semi-continuous function $ u $ is a viscosity
supersolution of \eqref{eqvissol} if for every $ \psi \in C^2(\Om\times(0,T])$ such that $ \psi $
touches $u$ at $(x,t) \in \Omega\times(0,T]$ strictly from below (that is, $u-\psi$ has a strict minimum at $(x,t)$ with $u(x,t) = \psi(x,t)$), we have
$$P^* \Big(x,t,\psi(x,t),\frac{\partial \psi}{\partial t}(x,t),D \psi(x,t),D^2\psi(x,t)\Big)\geq 0.$$

Finally, $u$ is a viscosity solution of \eqref{eqvissol} if it is both a sub- and supersolution.
\end{definition}

Now, for the second equation in \eqref{ED1} we introduce the following definition according to its elliptic nature
with $t$ as a parameter.
This time we fix a function
\[
Q:\Omega\times(0,T]\times\R\times\R^N\times\mathbb{S}^N\to\R.
\]
Associated with $Q$, we consider the PDE 
\begin{equation}
\label{eqvissol.2}
Q\Big(x,t,v (x,t), Dv (x,t), D^2v (x,t) \Big) =0, \qquad x \in \Omega, \ t\in(0,T).
\end{equation}
In our system \eqref{ED1} the second equation is given by
\begin{equation}
\label{eqvissol.2.999}
Q\Big(x,t,r, p, M \Big) = -   \displaystyle \frac{\kappa}{2} trace(M) + r - u(x,t)=0.
\end{equation}

Notice that there is no time derivative involved, but $t$ is still present in the operator $Q$. 
We will assume that $P$ satisfies a monotonicity property,
\[
X\leq Y \text{ in } \mathbb{S}^N \implies Q(x,t,r,p,X)\geq Q(x,t,r,p,Y)
\]
for all $(x,t,r,p)\in \Omega\times(0,T]\times\R\times\R^N$.
Here we have an equation that involves the Laplacian that is well defined, so there is no need to
consider upper and lower semicontinuous envelopes of $Q$. Now we have

\begin{definition}
\label{def.sol.viscosa.2}
An upper semi-continuous function $v$ is a  viscosity subsolution of \eqref{eqvissol.2} if
for every $ \phi \in C^2(\Omega)$ such that $\phi$ touches $v(\cdot,t)$ at $ x \in
\Omega$ strictly from above (that is, $u(\cdot,t)-\phi (\cdot)$ has a strict maximum at $x$ with $u(x,t) = \phi(x)$), we have
$$Q\Big(x,t,\phi(x), D \phi(x),D^2\phi(x) \Big)\leq 0.$$

A lower semi-continuous function $ u $ is a viscosity
supersolution of \eqref{eqvissol.2} if for every $ \psi \in C^2(\Omega)$ such that $ \psi $
touches $v(\cdot,t)$ at $x \in \Omega $ strictly from below (that is, 
$v(\cdot,t)-\psi(\cdot)$ has a strict minimum at $x$ with $u(x,t) = \psi(x)$), we have
$$Q \Big(x,t,\psi(x,t),D \psi(x,t),D^2\psi(x,t)\Big)\geq 0.$$

Finally, $v$ is a viscosity solution of \eqref{eqvissol} if it is both a sub- and supersolution.
\end{definition}

As we mentioned before, to deal with our system \eqref{ED1}, given a pair of continuous functions $u$, $v$, we just consider
\eqref{eqvissol.888} and \eqref{eqvissol.2.999},
\begin{equation} \label{P,Q}
\begin{array}{l}
\displaystyle P (x,t,r,s,p,X) = s- \langle X \frac{p}{|p|}, \frac{p}{|p|} \rangle +r - v(x,t) \\[10pt]
\displaystyle Q (x,t,r,p,X) = - \frac{\kappa}{2} trace (X) + r - u(x,t),
\end{array}
\end{equation}
and use them in Definitions \ref{def.sol.viscosa.1} 
and \ref{def.sol.viscosa.2}. That is, we understand that the pair $(u,v)$ is a viscosity solution
to \eqref{ED1} if $u$ is a viscosity solution to the first equation (in the sense of Definition 
 \ref{def.sol.viscosa.1} with $v(x,t)$ as a given function), and $v(x,t)$ is a solution to the second equation
 (this time in the sense of Definition \ref{def.sol.viscosa.2} with $u(x,t)$ in the right hand side). 
 
 \subsection{Probability. The Optional Stopping Theorem.}
We briefly recall (see \cite{Williams}) that a sequence of random variables
$\{M_{k}\}_{k\geq 1}$ is called a supermartingale (a submartingale) if
$$ \E[M_{k+1}\arrowvert M_{0},M_{1},...,M_{k}]\leq M_{k} \ \ (\geq)$$
Then, the Optional Stopping Theorem, that we will call {\it (OSTh)} in what follows, says:
given $\tau$ a stopping time such that one of the following conditions hold,
\begin{itemize}
\item[(a)] The stopping time $\tau$ is bounded a.s.;
\item[(b)] It holds that $\E[\tau]<\infty$ and there exists a constant $c>0$ such that $$\E[M_{k+1}-M_{k}\arrowvert M_{0},...,M_{k}]\leq c;$$
\item[(c)] There exists a constant $C>0$ such that $|M_{\min \{\tau,k\}}|\leq C$ a.s. for every $k$.
\end{itemize}
Then 
$$ \E[M_{\tau}]\leq \E [M_{0}] \ \ (\geq)$$
if $\{M_{k}\}_{k\geq 0}$ is a supermartingale (submartingale).
For the proof of this result we refer to \cite{Doob,Williams}.

\section{Description of the game}
\label{sect-Game}

The rules of the game are the following:
the game starts with a token at an initial position $(x_0,t_0) \in \Omega\times (0,T]$, in one of the two boards. 
In the fist board, with probability $1-\eps^2$, 
the players play Tug-of-War as described in \cite{PSSW,MPR} (this game is associated
with the infinity Laplacian). Playing the Tug-of-War game, the players toss a fair coin and the winner
chooses a new position of the game with the restriction that $x_1 \in B_\eps (x_0)$.
Then the new position of the game goes to the point $(x_1, t_0 -\eps^2)$, this means that, if we play in the first board, the next position
of the game must be at the time level $t_1=t_0-\eps^2$ (or $t_1=0$ if $t_0<\eps^2$). We decrease time by $\eps^2$ only if we play
in the first board.
On the other hand, with probability $\eps^2$ the token jumps to the second board (at the same position $(x_0,t_0)$
(without changing the time in this case). Playing in the second board,
with, probability $1-\eps^2$ the token is moved at random (uniform probability) to some point 
$(x_1,t_0) \in B_\eps (x_0)\times (0,T]$ (we keep the time unchanged in the second board) 
and with probability $\eps^2$ the token jumps back to the first board
(without changing the spacial position nor the time). Notice that in the second board we stay at every play at the same time
(we only change the space variable $x$, but the time variable $t$ remains the same).
The game continues until the position of the token leaves the spacial domain or the time variable becomes less or equal than 0. 
If we are in the first situation (we leave $\Omega$ at some positive time) and the last position of the game is $(x_\tau,t_\tau)$ 
(with $x_\tau\in\R^N \setminus\Om$ and $t_\tau >0$) in the first board, then Player 1 gets the final payoff $\ol{f}(x_\tau,t_\tau)$ and Player 2 gets $- \ol{f}(x_\tau,t_\tau)$
(notice that this is a zero sum game). But if the game stops due to the time becoming less than 0 (remark that this can only happens 
playing in the first board), the first player gets as final
payoff $u_0(x_\tau)$ when $x_\tau\in\Om$, or $\ol{f}(x_\tau,0)$ if $x_\tau\notin\Om$
and the second player gets minus this amount. 
Now, if we finish the game playing in the second board and $(x_\tau,t_\tau)$ is the last position (notice that $t_\tau$ must be
positive in this case), then Player 1 gets $\ol{g}(x_\tau,t_\tau)$ and Player 2  $- \ol{g}(x_\tau,t_\tau)$. We can think
this final payoff as Player 2 pays to Player 1 the amount given by the payoff functions according to the board in which the game ends and/or to the time when the game ends.
We have that the game generates a sequence of states
$$
P=\Big\{ (x_{0},t_0, j_{0}),(x_{1},t_1,j_{1}),...,(x_{\tau},t_\tau, j_{\tau}) \Big\}
$$
with $j_{i}\in\{1,2\}$ (this index gives the board in which we are playing) and $(x_{i},t_i)$ gives the position (both in space and time) 
in the board $j_{i}$. 
The dependence of the position of the token 
in one of the boards, $j_i$, will be made explicit only when 
needed. Also remark that the number of plays until the game ends, $\tau$, is finite almost surely
(that is, the game ends with probability one in at most a finite number of plays). 

A strategy $S_\I$ for Player~I is a function defined on the
partial histories that gives the next position of the game provided Player $\I$ wins the coin toss
(and the token is and stays in the first board)
\[
S_\I{\left((x_0,t_0,j_{0}),(x_1,t_1,j_{1}),\ldots,(x_n,t_n,1)\right)}= (x_{n+1},t_{n+1},1) 
\]
\[
\mbox{with } x_{n+1} \in B_\eps (x_n) \quad \mbox{and } t_{n+1}=t_n -\eps^2 >0 \quad \mbox{or } t_{n+1}=0 \quad \mbox{if } t_n -\eps^2 \leq 0
\]
Analogously, a strategy $S_\II$ for Player~II is a function defined on the
partial histories that gives the next position of the game provided Player $\II$ is who wins the coin toss
(and the token stays at the first board).

When the two players fix their strategies $S_I$ and $S_{II}$ we can compute the expected outcome as follows:
Given the sequence $(x_0,t_0),\ldots,(x_n,t_n)$ with $x_k\in\Om$ and $t_n>0$, if $(x_k,t_k)$ belongs to the first board, the next game position is distributed according to
the probability
\[
\begin{array}{l}
\displaystyle 
\pi_{S_\I,S_\II,1}((x_0,t_0,j_0),\ldots,(x_k,t_k,1),{A}, B)= \frac{1-\eps^2}{2} \delta_{S_\I ((x_0,t_0,j_0),\ldots,(x_k,t_k,1))} (A) +
\frac{1-\eps^2}{2} \delta_{S_\II ((x_0,t_0,j_0),\ldots,(x_k,t_k,1))} (A) \\[10pt]
\qquad \qquad\qquad \qquad\qquad \qquad\qquad \qquad \qquad \displaystyle
+ \eps^2 \delta_{(x_k,t_k)} (B).
\end{array}
\]
Here $A$ is a subset in the first board while $B$ is a subset in the second board. 
If 
$x_k$ belongs to the second board, the next game position is distributed according to
the probability
\[
\pi_{S_\I,S_\II,2}((x_0,t_0,j_0),\ldots,(x_k,t_k,2),{A}, B)= (1-\eps^2) U(B_\eps(x_k)) (B) +
\eps^2 \delta_{(x_k,t_k)} (A).
\]

By using the Kolmogorov's extension theorem and the one step transition probabilities, we can build a
probability measure $\mathbb{P}^{x_0}_{S_\I,S_\II}$ on the
game sequences (taking onto account the two boards). The expected payoff, when starting from $(x_0,t_0,j_0)$ and
using the strategies $S_\I,S_\II$, is
\begin{equation}
\label{eq:defi-expectation}
\mathbb{E}_{S_{\I},S_{\II}}^{(x_0,t_0,j_0)} [ h (x_\tau,t_\tau) ]=\int_{H^\infty} h (x_\tau,t_\tau)  \,  d
\mathbb{P}^{(x_0,t_0)}_{S_\I,S_\II}
\end{equation}
(here we use $h=f$ if $(x_\tau,t_\tau)$ is in the first board and $t_\tau >0$, $h=g$ if $(x_\tau,t_\tau)$ is in the second board and $t_\tau >0$, or finally $h=u_0 $ if $t_\tau \leq 0$ ).

The \emph{value of the game for Player I} is given by
\[
u^\eps_\I(x_0,t_0)=\inf_{S_{\I}} \sup_{S_{\II}}\,
\mathbb{E}_{S_{\I},S_{\II}}^{(x_0,t_0,1)}\left[h (x_\tau,t_\tau) \right]
\]
for $(x_0,t_0)\in\Omega\times (0,T]$ in the first board ($j_0 =1$), and by 
\[
v^\eps_\I(x_0,t_0)=\inf_{S_\I}\sup_{S_{\II}}\,
\mathbb{E}_{S_{\I},S_\II}^{(x_0,t_0,2)}\left[h (x_\tau,t_\tau) \right]
\]
for $(x_0,t_0)\in\Omega\times (0,T]$ in the second board ($j_0 =2$).

The \emph{value of the game for Player II} is given by the same formulas just reversing the $\inf$--$\sup$,
\[
u^\eps_\II(x_0,t_0)=\sup_{S_{\II}}\inf_{S_\I}\,
\mathbb{E}_{S_{\I},S_\II}^{(x_0,t_0,1)}\left[h (x_\tau,t_\tau) \right], 
\]
for $x_0$ in the first board and 
\[
v^\eps_\II(x_0,t_0)=\sup_{S_{\II}}\inf_{S_\I}\,
\mathbb{E}_{S_{\I},S_\II}^{(x_0,t_0,2)}\left[h (x_\tau,t_\tau) \right], 
\]
for $x_0$ in the second board.

Intuitively, the values $u_\I(x_0,t_0)$ and $u_\II(x_0,t_0)$ are the best
expected outcomes each player can guarantee when the game starts at
$(x_0,t_0)$ in the first board while $v_\I(x_0,t_0)$ and $v_\II(x_0,t_0)$ are the best
expected outcomes for each player in the second board.
If these values coincide $u^\eps_\I= u^\eps_\II$ and $v^\eps_\I= v^\eps_\II$, we say that the game has a value.

Before proving that the game has a value,
let us observe that the game ends almost surely no matter the strategies used by the players, that is $\P (\tau =+\infty) =0$, and therefore
the expectation \eqref{eq:defi-expectation} is well defined. This is because we cannot play infinitely in the same board, and playing in 
the first board we will leave $\Om \times [0,t)$ either exiting $\Omega$ or exhausting time (in a finite number of plays).

\section{Existence and uniqueness for the DPP}
\label{sect-DPP}

To see that the game has a value, we first observe that we have existence
of $(u^\eps, v^\eps)$, a pair of functions that satisfies the DPP.
The existence of such a pair can be obtained by Perron's method. In fact, let us start considering the following set (that is composed by pairs of
functions 
that are sub solutions to our DPP). Let
\begin{equation} \label{kkk}
C=\max \Big\{ \|\ol{f}\|_\infty , \|\ol{g}\|_\infty , \| u_0 \|_\infty
\Big\},
\end{equation}
and take 
\begin{equation}
\label{A}
{A} =\displaystyle \Big\{ (z^{\eps},w^{\eps}) : \mbox{ are bounded above by $C$ and verify (\textbf{e}) } \Big\},
\end{equation}
with
\begin{equation}
\label{e} \tag{\bf{e}}
\displaystyle \left\lbrace
\begin{array}{ll}
 \displaystyle z^{\eps}(x,t)\leq\eps^{2}w^{\eps}(x,t)+(1-\eps^{2})\Big\{\half \sup_{y \in B_{\eps}(x)}z^{\eps}(y,t-\eps^2) + \half \inf_{y \in B_{\eps}(x)}z^{\eps}(y,t-\eps^2)\Big\} &  \ x \in \Omega,t\in(0,T], \\[10pt]
   \displaystyle w^{\eps}(x,t)\leq\eps^{2}z^{\eps}(x,t)+(1-\eps^{2})\kint_{B_{\eps}(x)}w^{\eps}(y,t)dy  &  \ x \in \Omega, t\in(0,T],  \\[10pt]
z^{\eps}(x,t) \leq \ol{f}(x,t)  & \ x \in \R^{N} \backslash \Omega, t > 0,  \\[10pt]
w^{\eps}(x,t) \leq \ol{g}(x,t)  & \ x \in \R^{N} \backslash \Omega, t > 0, \\[10pt] 
z^{\eps}(x,t)\leq u_0(x)  & \ x \in \Om , t\leq 0.
\end{array}
\right.
\end{equation}

Observe that 
 ${A} \neq \emptyset$. To see this fact, we just take 
 $z^{\eps}=-C$ and $ w^{\eps}=-C$
 with $C$ given by \eqref{kkk}.
Now we let
 \begin{equation}
 \label{u}
 u^{\eps}(x,t)=\sup_{(z^{\eps},w^{\eps})\in {A}}z^{\eps}(x,t) 
 \qquad \mbox{and} \qquad
   v^{\eps}(x,t)=\sup_{(z^{\eps},w^{\eps})\in {A}}w^{\eps}(x,t) .
   \end{equation}
Our goal is to show that in this way we find a solution to the DPP.

\begin{proposition} \label{prop.DPP.tiene.sol}
The pair $(u^{\eps},v^{\eps})$ given by \eqref{u} is a solution to the DPP \eqref{DPP}.
\end{proposition}

\begin{proof}
First, let us see that $(u^{\eps},v^{\eps})$ belongs to the set ${A}$. To this end we first observe that 
$u^{\eps}$ y $v^{\eps}$ are bounded by $C$ and verify 
 $u^{\eps}(x,t)\leq \ol{f}(x,t)$ , $v^{\eps}(x,t)\leq \ol{g}(x,t)$ for $x\in\R^{N}\backslash\Om$, $t\geq 0$, and $z^{\eps}(x,t)\leq u_0(x)$ for $x\in\Om$, $t\leq 0$. 
 Hence we need to check \eqref{e} for $x\in\Om$ , $t>0$. Take $(z^{\eps},w^{\eps})\in {A}$ and fix $(x,t)\in\Om\times(0,T]$,. Then,
$$z^{\eps}(x,t)\leq\eps^{2}w^{\eps}(x,t)+(1-\eps^{2})\Big\{\half \sup_{y \in B_{\eps}(x)}z^{\eps}(y,t-\eps^2) + \half \inf_{y \in B_{\eps}(x)}z^{\eps}(y,t-\eps^2)\Big\}.$$
As $z^{\eps}\leq u^{\eps}$ and $w^{\eps}\leq v^{\eps}$ we obtain 
$$z^{\eps}(x,t)\leq\eps^{2}v^{\eps}(x,t)+(1-\eps^{2})\Big\{\half \sup_{y \in B_{\eps}(x)}u^{\eps}(y,t-\eps^2) + \half \inf_{y \in B_{\eps}(x)}u^{\eps}(y,t-\eps^2)\Big\}.$$
Taking supremum in the left hand sise we obtain
$$u^{\eps}(x,t)\leq\eps^{2}v^{\eps}(x,t)+(1-\eps^{2})\Big\{\half \sup_{y \in B_{\eps}(x)}u^{\eps}(y,t-\eps^2) + \half \inf_{y \in B_{\eps}(x)}u^{\eps}(y,t-\eps^2)\Big\}$$
In an analogous way we obtain  
$$ v^{\eps}(x,t)\leq\eps^{2}u^{\eps}(x,t)+(1-\eps^{2})\kint_{B_{\eps}(x)}v^{\eps}(y,t)dy,$$
and we conclude that $(u^{\eps},v^{\eps}) \in {A}$.

To end the proof we need to see that $(u^{\eps},v^{\eps})$ verifies the equalities in the equations in condition \eqref{e}. 
We argue by contradiction and assume that there is a point $(x_{0},t_0)\in\R^{N}\times[0,T]$ where an inequality in \eqref{e} is strict.
First, assume that $(x_{0},t_0)\in\R^{N}\backslash\Om\times[0,T]$ and that we have $u^{\eps}(x_{0},t_0)<\ol{f}(x_{0},t_0)$, or the case $x_0\in\Om$, $t_0\leq 0$. and $z^{\eps}(x_0,t_0)<u_0(x_0)$.
Then, take $u^{\eps}_{0}$ defined by $u^{\eps}_{0}(x,t)=u^{\eps}(x,t)$ for $(x,t)\neq (x_0,t_0)$ and $u^{\eps}_{0}(x_{0},t_0)=\ol{f}(x_{0},t_0)$ in the first case, and $u^{\eps}_0(x_0,t_0)=u_0(x_0)$ in the second case.
The pair $(u^{\eps}_{0},v^{\eps})$ belongs to ${A}$ but $u^{\eps}_{0}(x_{0})>u^{\eps}(x_{0})$
which is a contradiction. We can argue is a similar way if $v^{\eps}(x_{0},t_0)<\ol{g}(x_{0},t_0)$. 
Next, we consider a point $(x_{0},t_0)\in\Om\times(0,T]$ with one of the inequalities in \textbf{e} strict. Assume that 
$$u^{\eps}(x_{0},t_0)<\eps^{2}v^{\eps}(x_{0},t_0)+(1-\eps^{2})\Big\{\half \sup_{y \in B_{\eps}(x_{0})}u^{\eps}(y,t_0-\eps^2) + \half \inf_{y \in B_{\eps}(x_{0})}u^{\eps}(y,t_0-\eps^2)\Big\}.$$
Let
$$ \delta=\eps^{2}v^{\eps}(x_{0},t_0)+(1-\eps^{2})\Big\{\half \sup_{y \in B_{\eps}(x_{0})}u^{\eps}(y,t_0-\eps^2) + \half \inf_{y \in B_{\eps}(x_{0})}u^{\eps}(y,t_0-\eps^2)\Big\}-u^{\eps}(x_{0},t_0)>0, $$
and consider the function $u^{\eps}_{0}$ given by;
$$
 u^{\eps}_{0} (x,t) =\left\lbrace
 \begin{array}{ll}
 u^{\eps}(x,t) &  \ \ (x,t) \neq (x_{0},t_0),  \\[5pt]
 \displaystyle u^{\eps}(x,t)+\frac{\delta}{2} &  \ \ (x,t) =(x_{0},t_0) . \\
 \end{array}
 \right.
$$
Observe that
$$u^{\eps}_{0}(x_{0},t_0)=u^{\eps}(x_{0},t_0)+\frac{\delta}{2}<\eps^{2}v^{\eps}(x_{0},t_0)+(1-\eps^{2})\Big\{\half \sup_{y \in B_{\eps}(x_{0})}u^{\eps}(y,t_0-\eps^2) + \half \inf_{y \in B_{\eps}(x_{0})}u^{\eps}(y,t_0-\eps^2)\Big\}$$
and hence
$$u^{\eps}_{0}(x_{0},t_0)<\eps^{2}v^{\eps}(x_{0},t_0)+(1-\eps^{2})\Big\{\half \sup_{y \in B_{\eps}(x_{0})}u^{\eps}_{0}(y,t_0-\eps^2) + \half \inf_{y \in B_{\eps}(x_{0})}u^{\eps}_{0}(y,t_0-\eps^2)\Big\}.$$
Then we have that $(u^{\eps}_{0},v^{\eps})\in {A}$ but $u^{\eps}_{0}(x_{0}t_0)>u^{\eps}(x_{0},t_0)$ reaching again a contradiction.

In an analogous way we can show that when 
$$v^{\eps}(x_{0},t_0)<\eps^{2}u^{\eps}(x_{0},t_0)+(1-\eps^{2})\kint_{B_{\eps}(x_{0})}v^{\eps}(y,t_0)dy,$$
we also reach a contradiction.
\end{proof}

Now, concerning the value functions of our game, we know that $u^\eps_\I\geq u^\eps_\II$ 
and $v^\eps_\I\geq v^\eps_\II$
(this is immediate from the definitions). Hence, 
to obtain uniqueness of solutions of the DPP and existence of value functions for our game, 
it is enough to show that $u^\eps_\II \geq u^\eps\geq u^\eps_\I$ and 
$v^\eps_\II \geq v^\eps\geq v^\eps_\I$. To show this result
we will use the \textit{OSTh} for sub/supermartingales
(see Section \ref{sect-prelim}).

\begin{theorem}
Gigen $\eps>0$ let $(u^{\eps},v^{\eps})$ a pair of functions that verifies the DPP \eqref{DPP}, then it holds that
$$
u^{\eps}(x_{0},t_0)=\sup_{S_{I}}\inf_{S_{II}}\E^{(x_{0},t_0,1)}_{S_{I},S_{II}}[h( x_{\tau},t_\tau)] 
$$
if $(x_{0},t_0) \in \Om\times (0,T)$ is in the first board and
$$
v^{\eps}(x_{0},t_0)=\sup_{S_{I}}\inf_{S_{II}}\E^{(x_{0},t_0,2)}_{S_{I},S_{II}}[h( x_{\tau},t_\tau)]
$$
 if $(x_{0},t_0) \in \Om\times (0,T)$ is in the second board. 
 
 Moreover, we can interchange $\inf$ with $\sup$ in the previous identities, that is, the game has a value. 
 This value can be characterized as the unique solution to the DPP.  
\end{theorem}

\begin{proof}
Given $\eps>0$ we have proved the existence of a solution to the DPP $(u^{\eps},v^{\eps})$. Fix $\delta>0$. 
Assume that we start with $(x_{0},t_0,1)$, that is, the initial position is at board 1. We choose a strategy for 
Player I as follows:
$$
(x_{k+1}^{I},t_{k}^{I})=S_{I}^{*}((x_{0},t_0),...,(x_{k},t_k)) 
\qquad 
\mbox{is such that} \qquad
\sup_{y\in B_{\eps}(x_{k})}u^{\eps}(y,t-\eps^2) - \frac{\delta}{2^{k}}\leq u^{\eps}(x_{k+1}^{I},t_{k+1}^{I}).
$$
Given this strategy for Player I and any strategy $S_{II}$ for Player II we consider the sequence of random variables
given by
$$
M_{k}=\left\lbrace
 \begin{array}{ll}
 \displaystyle u^{\eps}(x_{k},t_k)-\frac{\delta}{2^{k}} & \ \ \mbox{if } \ (j_{k}=1),  \\[10pt]
 \displaystyle v^{\eps}(x_{k},t_k)-\frac{\delta}{2^{k}} & \ \ \mbox{if } \ (j_{k} = 2).
 \end{array}
 \right.
$$
Let us see that $(M_{k})_{\kappa\geq 0}$ is a submartingale. 
To this end we need to estimate 
$$
\E^{(x_{0},t_0,1)}_{S_{I}^{*},S_{II}}[M_{k+1}\arrowvert M_{k}]=\E^{(x_{0},t_0,1)}_{S_{I}^{*},S_{II}}[M_{k+1}\arrowvert (x_{k},t_k,j_{k})].
$$
We consider two cases:

\textbf{Case 1:} Assume that $j_{k}=1$, then 
$$
\begin{array}{ll}
\displaystyle 
\E^{(x_{0},t_0,1)}_{S_{I}^{*},S_{II}}[M_{k+1}\arrowvert (x_{k},t_k,1)] \\[10pt]
\qquad \displaystyle = (1-\eps^{2})\E^{(x_{0},t_0,1)}_{S_{I}^{*},S_{II}}[M_{k+1}\arrowvert (x_{k},t_k,1)\wedge j_{k+1}=1]+\eps^{2}\E^{(x_{0},t_0,1)}_{S_{I}^{*},S_{II}}[M_{k+1}\arrowvert (x_{k},t_k,1)\wedge j_{k+1}=2].
\end{array}
$$

Her we used that the probability of staying in the same board is $(1-\eps^{2})$ and the probability of jumping to the other board is
$\eps^{2}$. Now, if $j_{k}=1$ and $j_{k+1}=2$ then $x_{k+1}=x_{k}$ and $t_{k+1}=t_k$ (we just changed boards). 
On the other hand, if we stay in the first board we obtain
$$
\E^{(x_{0},t_0,1)}_{S_{I}^{*},S_{II}}[M_{k+1}\arrowvert (x_{k},t_k,1)]=(1-\eps^{2})\Big\{\half u^{\eps}(x_{k+1}^{I},t_{k+1})+\half u^{\eps}(x_{k+1}^{II},t_{k+1})-\frac{\delta}{2^{k+1}}\Big\}+\eps^{2}(v^{\eps}(x_{k},t_k)-\frac{\delta}{2^{k+1}}).
$$
Since we are using the strategies $S_{I}^{*}$ and $S_{II}$, it holds that 
$$
\sup_{y\in B_{\eps}(x_{k})}u^{\eps}(y,t_k -\eps^2) - \frac{\delta}{2^{k}}\leq u^{\eps}(x_{k+1}^{I},t_{k+1})
\qquad
\mbox{and}\qquad
\inf_{y\in B_{\eps}(x_{k})}u^{\eps}(y,t-\eps^2)\leq u^{\eps}(x_{k+1}^{II},t_{k+1}).
$$
Therefore, we arrive to
$$
\begin{array}{l}
\displaystyle 
\E^{(x_{0},1)}_{S_{I}^{*},S_{II}}[M_{k+1}\arrowvert (x_{k},t_k,1)]
\\[10pt]
\qquad \displaystyle 
\geq(1-\eps^{2})\Big\{\half \sup_{y\in B_{\eps}(x_{k})}u^{\eps}(y,t_{k}-\eps^2 - \frac{\delta}{2^{k}})+\half \inf_{y\in B_{\eps}(x_{k})}u^{\eps}(y,t_{k}-\eps^2)\Big\} +\eps^{2}v^{\eps}(x_{k},t_k)-\frac{\delta}{2^{k+1}},
\end{array}
$$
that is,
$$
\begin{array}{l}
\displaystyle 
\E^{(x_{0},t_0,1)}_{S_{I}^{*},S_{II}}[M_{k+1}\arrowvert (x_{k},t_k,1)] 
 \displaystyle \\[10pt]
\qquad  \displaystyle \geq(1-\eps^{2})\Big\{\half \sup_{y\in B_{\eps}(x_{k})}u^{\eps}(y,t-\eps^2)+\half \inf_{y\in B_{\eps}(x_{k})}u^{\eps}(y,t-\eps^2)\Big\}+\eps^{2}v^{\eps}(x_{k},t_k)-(1-\eps^{2}) \frac{\delta}{2^{k+1}}-\frac{\delta}{2^{k+1}}.
\end{array}
$$
As $u^{\eps}$ is a solution to the DPP \eqref{DPP} we obtain 
$$
\E^{(x_{0},t_0,1)}_{S_{I}^{*},S_{II}}[M_{k+1}\arrowvert (x_{k},t_k,1)]\geq u^{\eps}(x_{k},t_k)-\frac{\delta}{2^{k}}=M_{k}
$$
as we wanted to show.

\textbf{Case 2:} Assume that $j_{k}=2$. With the same ideas used before we get
$$
\begin{array}{l}
\displaystyle 
\E^{(x_{0},t_0,1)}_{S_{I}^{*},S_{II}}[M_{k+1}\arrowvert (x_{k},t_k,2)] 
 \\[10pt] \qquad \displaystyle =(1-\eps^{2})\E^{(x_{0},t_0,1)}_{S_{I}^{*},S_{II}}[M_{k+1}\arrowvert (x_{k},t_k,2)\wedge j_{k+1}=2]+\eps^{2}\E^{(x_{0},1)}_{S_{I}^{*},S_{II}}[M_{k+1}\arrowvert (x_{k},t_k,2)\wedge j_{k+1}=1].
\end{array}
$$
Remark that when $j_{k}=j_{k+1}=2$ (this means that we play in the second board) with $x_{k}\in\Om$, then $x_{k+1}$ is chosen with
uniform probability in the ball $B_{\eps}(x_{k})$. Hence,
$$
\begin{array}{l}
\displaystyle \E^{(x_{0},t_0,1)}_{S_{I}^{*},S_{II}}[M_{k+1}\arrowvert (x_{k},t_k,2)\wedge j_{k+1}=2]
 \\[10pt] \displaystyle \qquad =\E^{(x_{0},t_0,1)}_{S_{I}^{*},S_{II}}[v^{\eps}(x_{k+1},t_{k+1})-\frac{\delta}{2^{k+1}}\arrowvert (x_{k},2)\wedge j_{k+1}=2]
\displaystyle =\kint_{B_{\eps}(x_{k})}v^{\eps}(y,t_k)dy-\frac{\delta}{2^{k+1}}.
\end{array}
$$
On the other hand, 
$$
\E^{(x_{0},t_0,1)}_{S_{I}^{*},S_{II}}[M_{k+1}\arrowvert (x_{k},t_k,2)\wedge j_{k+1}=1]=u^{\eps}(x_{k},t_k)-\frac{\delta}{2^{k+1}}.
$$
Collecting these estimates we obtain
$$
\begin{array}{l}
\displaystyle 
\E^{(x_{0},t_0,1)}_{S_{I}^{*},S_{II}}[M_{k+1}\arrowvert (x_{k},t_k,2)]=(1-\eps^{2})\left(\kint_{B_{\eps}(x_{k})}v^{\eps}(y,t_k)dy-\frac{\delta}{2^{k+1}}
\right)+\eps^{2} \Big(u^{\eps}(x_{k},t_k)-\frac{\delta}{2^{k+1}} \Big) \\[10pt]
\qquad \displaystyle \geq(1-\eps^{2})\kint_{B_{\eps}(x_{k})}v^{\eps}(y,t_k)dy+\eps^{2}u^{\eps}(x_{k},t_k)-\frac{\delta}{2^{k}},
\end{array}
$$
that is,
$$
\E^{(x_{0},t_0,1)}_{S_{I}^{*},S_{II}}[M_{k+1}\arrowvert (x_{k},t_k,2)]\geq v^{\eps}(x_{k},t_k)-\frac{\delta}{2^{k}}=M_{k}.
$$
Here we used that $v^{\eps}$ is a solution to the DPP, \eqref{DPP}. This ends the second case.

Therefore $(M_{k})_{k\geq 0}$ is a \textit{submartingale}. Using the \textit{OSTh} 
(recall that we have proved that $\tau$ is finite a.s. and that we have that $M_k$ is uniformly bounded) we conclude that 
$$
\E^{(x_{0},t_0,1)}_{S_{I}^{*},S_{II}}[M_{\tau}]\geq M_{0}
$$
where $\tau$ is the first time such that $(x_{\tau},t_\tau)\notin\Om\times (0,T]$ in any of the two boards. Then, 
$$
\E^{(x_{0},t_0,1)}_{S_{I}^{*},S_{II}}[\mbox{final payoff}]\geq u^\eps (x_{0},t_0)-\delta.
$$
We can compute the infimum  in $S_{II}$ and then
the supremum in $S_{I}$ to obtain
$$
\sup_{S_{I}}\inf_{S_{II}}\E^{(x_{0},t_0,1)}_{S_{I},S_{II}}[\mbox{final payoff}]\geq u^{\eps}(x_{0},t_0)-\delta.
$$

We just observe that if we have started in the second board the previous computations show that
$$
\sup_{S_{I}}\inf_{S_{II}}\E^{(x_{0},t_0,2)}_{S_{I},S_{II}}[\mbox{final payoff}]\geq v^{\eps}(x_{0},t_0)-\delta.
$$

Now our goal is to prove the reverse inequality (interchanging inf and sup). 
To this end we define an strategy for Player II with
$$
(x_{k+1}^{II},t_{k+1})=S_{II}^{*}\big((x_{0},t_0),...,(x_{k},t_k)\big)
\qquad
\mbox{is such that} 
\qquad
\inf_{y\in B_{\eps}(x_{k})}u^{\eps}(y,t_{k+1})+\frac{\delta}{2^{k}}\geq u^{\eps}(x_{k+1}^{II},t_{k+1}),
$$
and consider the sequence of random variables 
$$
N_{k}=\left\lbrace
 \begin{array}{ll}
 \displaystyle u^{\eps}(x_{k},t_k)+\frac{\delta}{2^{k}} & \ \ \mbox{if }  j_{k}=1  \\[10pt]
 \displaystyle v^{\eps}(x_{k},t_k)+\frac{\delta}{2^{k}} & \ \ \mbox{if }   j_{k}=2. 
 \end{array}
 \right.
$$
Arguing as before we obtain that this sequence is a \textit{supermartingale}.
From the \textit{OSTh} we get 
$$
\E^{(x_{0},t_0,1)}_{S_{I}^{*},S_{II}}[N_{\tau}]\leq N_{0}
$$
where $\tau$ is the stopping time for the game. Then, 
$$
\E^{(x_{0},t_0,1)}_{S_{I},S_{II}^{*}}[\mbox{final payoff}]\leq u^\eps (x_{0},t_0)+\delta.
$$
Taking supremum in $S_{I}$ and then infimum in $S_{II}$ we obtain
$$
\inf_{S_{II}}\sup_{S_{I}}\E^{(x_{0},t_0,1)}_{S_{I},S_{II}}[\mbox{final payoff}]\leq u^{\eps}(x_{0},t_0)+\delta.
$$
As before, the same ideas starting at $(x_{0},t_0,2)$ give us
$$
\inf_{S_{II}}\sup_{S_{I}}\E^{(x_{0},1)}_{S_{I},S_{II}}[\mbox{final payoff}]\leq v^{\eps}(x_{0})+\delta.
$$

To end the proof we just observe that 
$$
\sup_{S_{I}}\inf_{S_{II}}\E_{S_{I},S_{II}}[\mbox{final payoff}]\leq \inf_{S_{II}}\sup_{S_{I}}\E_{S_{I},S_{II}}[\mbox{final payoff}].
$$
Therefore,
$$
u^{\eps}(x_{0},t_0)-\delta\leq\sup_{S_{I}}\inf_{S_{II}}\E_{S_{I},S_{II}}^{(x_{0},t_0,1)}[\mbox{final payoff}]
\leq \inf_{S_{II}}\sup_{S_{I}}\E_{S_{I},S_{II}}^{(x_{0},t_0,1)}[\mbox{final payoff}]\leq u^{\eps}(x_{0},t_0)+\delta
$$
and
$$
v^{\eps}(x_{0},t_0)-\delta\leq\sup_{S_{I}}\inf_{S_{II}}\E_{S_{I},S_{II}}^{(x_{0},t_0,2)}[\mbox{final payoff}]\leq 
\inf_{S_{II}}\sup_{S_{I}}\E_{S_{I},S_{II}}^{(x_{0},t_0,2)}[\mbox{final payoff}]\leq v^{\eps}(x_{0},t_0)+\delta.
$$
As $\delta >0$ is arbitrary the proof is finished.
\end{proof}

\section{Uniform convergence}
\label{sect-uniform}

Now our aim is to pass to the limit in the values of the game and prove that
$$
u^\eps \to u \ \mbox{ and }  \ v^\eps \to v, \qquad \mbox{as } \eps \to 0,
$$
uniformly in $\overline{\Omega} \times [0,T]$,
and then in the next section to obtain that this limit pair $(u,v)$ is a viscosity solution to our system \eqref{ED1}.

To obtain a convergent subsequence we will use the following
Arzela-Ascoli type lemma both for $u^\eps$ and for $v^\eps$. For its proof see Lemma~4.2 from \cite{MPRb}.

\begin{Lemma}\label{lem.ascoli.arzela} Let $\{w^\eps : \overline{\Omega} \times [0,T]
\to \R,\ \eps>0\}$ be a set of functions such that
\begin{enumerate}
\item there exists $C>0$ such that $|{w^\eps (x,t)}|<C$ for
    every $\eps >0$ and every $x \in \overline{\Omega}$, $t\in [0,T]$,
\item \label{cond:2} given $\delta >0$ there are constants
    $r_0$ and $\eps_0$ such that for every $\eps < \eps_0$
    and any $x, y \in \overline{\Omega}$ with $|x - y | < r_0 $ and any $t,s \in [0,T]$ with
    $|t-s|<r_0$,
    it holds
$$
|w^\eps (x,t) - w^\eps (y,s)| < \delta.
$$
\end{enumerate}
Then, there exists  a uniformly continuous function $w:
\overline{\Omega} \times [0,T] \to \R$ and a subsequence still denoted by
$\{w^\eps \}$ such that
\[
\begin{split}
w^{\eps}\to w \qquad\textrm{ uniformly in }\overline{\Omega} \times [0,T],
\mbox{ as $\eps\to 0$.}
\end{split}
\]
\end{Lemma}

So our task now is to show that $u^\eps$ and $v^\eps$ both satisfy the hypotheses of the previous lemma. First, we
observe that they are uniformly bounded.

\begin{Lemma}\label{lem.ascoli.arzela.acot} 
There exists a constant $C>0$ independent of $\eps$ such that $$|u^\eps (x,t)|\leq C, \qquad |v^\eps (x,t)|\leq C,$$ for
    every $\eps >0$ and every $(x,t) \in \overline{\Omega} \times [0,T]$.
\end{Lemma}

\begin{proof} 
We observe that we can take
$$
C = \max \{ \| g\|_\infty, \| f \|_\infty, \|u_0\|_\infty \},
$$
since the final payoff in any of the boards is bounded by this constant $C$.
\end{proof}

To prove the second hypothesis of Lemma \ref{lem.ascoli.arzela} we will need 
some key estimates according to the board in which we are playing.

\subsection{Estimates for the Tug-of-War game}

In this case we are going to assume that we are permanently playing in board 1 (with the Tug-of-War game).
We introduce the notations $\Om_T =\Om\times (0,T]$ for a parabolic cilinder and $\partial_p \Om_T =\big[\partial\Om\times (0,T]\big]\cup \big[\, \ol{\Om}\times \{ 0 \}\big]$ for
the parabolic boundary.

\begin{Lemma}\label{lema.estim.ToW}
	Given $\eta>0$ and $a>0$, there exist $r_{0}>0$, $l_0>0$ and $\eps_{0}>0$ such that, given $(y,s)\in\partial_p \Om_T$ and $(x_{0},t_0)\in\Om_T$ with $| x_{0}-y |<r_{0}$, $|t_0-s|<l_0$ 
	any of the two players has a strategy $S^{*}$ with which we obtain 
	$$\mathbb{P}\big(| x_{\tau}-y | < a \big) \geq 1 - \eta \quad \mbox{,} \quad \P\big(|t_{\tau}-s|< a \big)\geq 1-\eta \quad \mbox{and} \quad
	\mathbb{P}\Big(\tau \geq \frac{a}{2\eps^2}\Big)< \eta$$
	for $\eps<\eps_{0}$ and here $(x_{\tau},t_{\tau})$ denotes the first position outside $\Om\times (0,T]$.
\end{Lemma}

\begin{proof}
	We consider two cases depending of the position of $(y,s)$ in $\partial_p \Om_T=\big[\partial\Om\times (0,T]\big]\cup \big[\, \ol{\Om}\times \{ 0 \}\big]$. 
	
	\textbf{Case 1:} If $(y,s)\in\partial\Om\times (0,T]$. 
	We can assume without loss of generality that $y=0\in\partial\Om$. In this case we consider the strategy $S^{*}$ (this strategy can be used
	by any of the two players) that is given by \lq\lq point to the point $y=0$''. This strategy is given by
	$$x_{k+1}=S^{*}\big(x_{0},x_{1},...,x_{k}\big)=x_{k}+\Big(\frac{\eps^{3}}{2^{k}}-\eps \Big)\frac{x_{k}}{ |x_{k} |},$$
	when $|x_k|\geq \eps$, in the other case we take $x_{k+1}=0$. Notice that the strategy only depends of de position $x_k$, not of the time $t_k$ 
	(that is deterministic, $t_{k+1}=t_k-\eps^2$, and then not affected by the choice of the strategies). 
	Now, let us consider the random variables
	$$
	N_{k}=
	|x_{k}|+\frac{\eps^{3}}{2^{k}}
	$$
	for $k\geq 0$ and play assuming that one of the players uses the $S^{*}$ strategy. The goal is to prove that $\{N_{k}\}_{k\geq 0}$ 
	is \textit{supermartingale}, i.e.,
	$$ \E[N_{k+1}\arrowvert N_{k}]\leq N_{k}. $$
	Note that with probability 1/2 we obtain
	$$x_{k+1}=x_{k} + \Big(\frac{\eps^{3}}{2^{k}}-\eps \Big)\frac{x_{k}}{ | x_{k} |}$$
	this is the case when the player who uses the $S^{*}$ strategy wins the coin toss. On the other hand, we have 
	$$| x_{k+1}|\leq | x_{k}| + \eps, $$ 
	when the other player wins (no matter what strategy he uses).
	Then, we obtain
	$$ \E\Big[| x_{k+1}| \arrowvert  x_{k} \Big]\leq \half \Big(| x_{k}| +(\frac{\eps^{3}}{2^{k}}-\eps)\Big) + \half(| x_{k} | + \eps)= | x_{k} |
	+\frac{\eps^{3}}{2^{k+1}}.$$
	Hence, we get
	$$ \E \Big[N_{k+1}\arrowvert N_{k}\Big]=\E \Big[ | x_{k+1}| + \frac{\eps^{3}}{2^{k+1}}\arrowvert | x_{k} | + \frac{\eps^{3}}{2^{k}}\Big]
	\leq |x_{k}| + \frac{\eps^{3}}{2^{k+1}}+\frac{\eps^{3}}{2^{k+1}}=N_{k}.$$
	We just proved that $\{N_{k}\}_{k\geq 0}$ is a \textit{supermartingale}.
	Now, let us consider the random variables
	$$ (N_{k+1}-N_{k})^{2},$$
	and the event
	\begin{equation}
	\label{Fk}
	F_{k}=\{  the \ player \ who \ points \ to \ 0\in\partial\Om \ wins \ the \ coin \ toss  \}.
	\end{equation}
	Then we have the following
	$$
	\begin{array}{l}
	\displaystyle \E[(N_{k+1}-N_{k})^{2}\arrowvert N_{k}]
	\displaystyle =\half\E[(N_{k+1}-N_{k})^{2}\arrowvert N_{k} \wedge F_{k}]+\half\E[(N_{k+1}-N_{k})^{2}\arrowvert N_{k} \wedge F_{k}^{c}]
	\\[10pt]
	\qquad \displaystyle \geq 
	\displaystyle  \half\E[(N_{k+1}-N_{k})^{2}\arrowvert N_{k} \wedge F_{k}].
	\end{array}
	$$
	Let us observe that when $|x_k|\geq \eps$
	$$
	\begin{array}{l}
	\displaystyle
	\half\E[(N_{k+1}-N_{k})^{2}\arrowvert N_{k} \wedge F_{k}]
	\displaystyle =\half\E \Big[( |x_{k}| -\eps+\frac{\eps^{3}}{2^{k}}+\frac{\eps^{3}}{2^{k+1}}- | x_{k}| -\frac{\eps^{3}}{2^{k}})^{2} \Big]
	\displaystyle =\half\E \Big[(-\eps+\frac{\eps^{3}}{2^{k+1}})^{2} \Big]\geq\frac{\eps^{2}}{3}
	\end{array}
	$$
	if $\eps<\eps_{0}$ for $\eps_{0}$ small enough. In the case that $|x_k|<\eps$ we have that $x_{k+1}=0$, then
	$$
	\begin{array}{l}
	\displaystyle
	\half\E[(N_{k+1}-N_{k})^{2}\arrowvert N_{k} \wedge F_{k}]
	\displaystyle =\half\E \Big[( \frac{\eps^{3}}{2^{k+1}}- | x_{k}| -\frac{\eps^{3}}{2^{k}})^{2} \Big]
	\displaystyle >\half\E \Big[(-\eps-\frac{\eps^{3}}{2^{k+1}})^{2} \Big]\geq\frac{\eps^{2}}{3}
	\end{array}
	$$
	and we have obtained the same estimate.
	With this estimate in mind we get
	\begin{equation}
	\label{e/3b}
	\E[(N_{k+1}-N_{k})^{2}\arrowvert N_{k}] \geq \frac{\eps^{2}}{3}.
	\end{equation}
	Now, we analyze $N_{k}^{2}-N_{k+1}^{2}$. We have
	\begin{equation}
	\label{M2b}
	N_{k}^{2}-N_{k+1}^{2}=(N_{k+1}-N_{k})^{2}+2N_{k+1}(N_{k}-N_{k+1}).
	\end{equation}
	Let us prove that $\E[N_{k+1}(N_{k}-N_{k+1})\arrowvert N_{k}]\geq 0$ using the set $F_{k}$ defined by \eqref{Fk}. It holds that
	$$
	\begin{array}{l}
	\displaystyle
	\E[N_{k+1}(N_{k}-N_{k+1})\arrowvert N_{k}]
	\\[10pt]
	\qquad \displaystyle =\half\E[N_{k+1}(N_{k}-N_{k+1})\arrowvert N_{k}\wedge F_{k}]+\half\E[N_{k+1}(N_{k}-N_{k+1})\arrowvert N_{k}\wedge F_{k}^{c}]
	\\[10pt]
	\qquad \displaystyle =
	\half \Big[(|x_{k}|-\eps +\frac{\eps^{3}}{2^{k}}+\frac{\eps^{3}}{2^{k+1}})(|x_{k}|+\frac{\eps^{3}}{2^{k}}-|x_{k}|+\eps-\frac{\eps^{3}}{2^{k}}-\frac{\eps^{3}}{2^{k+1}})\Big] \\[10pt]
	\qquad \displaystyle \qquad +
	\half \Big[(|x_{k+1}|+\frac{\eps^{3}}{2^{k+1}})(|x_{k}|+\frac{\eps^{3}}{2^{k}}-|x_{k+1}|-\frac{\eps^{3}}{2^{k+1}})\Big]
	\\[10pt]
	\qquad \displaystyle 
	\geq
	\half \Big(|x_{k}|-\eps+\frac{\eps^{3}}{2^{k}}+\frac{\eps^{3}}{2^{k+1}}\Big)\Big(\eps-\frac{\eps^{3}}{2^{k+1}}\Big)
	\\[10pt]
	\qquad \displaystyle  \qquad +\half \Big[(|x_{k}|-\eps+\frac{\eps^{3}}{2^{k+1}})(|x_{k}|+\frac{\eps^{3}}{2^{k}}-
	|x_{k}|-\eps-\frac{\eps^{3}}{2^{k+1}})\Big],
	\end{array} 
	$$
	here we used that  $|x_{k}|-\eps\leq|x_{k+1}|\leq|x_{k}|+\eps$. Thus, we have
	$$\E[N_{k+1}(N_{k}-N_{k+1})\arrowvert N_{k}]\geq\half(|x_{k}|-\eps+\frac{\eps^{3}}{2^{k+1}}+\frac{\eps^{3}}{2^{k}})(\eps-\frac{\eps^{3}}{2^{k+1}})+\half(|x_{k}|-\eps+\frac{\eps^{3}}{2^{k+1}})(-\eps+\frac{\eps^{3}}{2^{k+1}}),$$
	and then, 
	$$\E[N_{k+1}(N_{k}-N_{k+1})\arrowvert N_{k}]\geq\half \Big[\frac{\eps^{3}}{2^{k}}(\eps-\frac{\eps^{3}}{2^{k+1}})\Big]\geq 0.$$
	If we go back to \eqref{M2b} and use \eqref{e/3b} and the result we just obtained we arrive to
	$$\E[N_{k}^{2}-N_{k+1}^{2}\arrowvert N_{k}]\geq\E[(N_{k+1}-N_{k})^{2}\arrowvert N_{k}]\geq \frac{\eps^{2}}{3}.$$
	For the sequence of random variables
	$$\W_{k}=N^{2}_{k}+\frac{k\eps^{2}}{3}$$
	we have
	$$\E[\W_{k}-\W_{k+1}\arrowvert \W_{k}]=\E[N_{k}^{2}-N_{k+1}^{2}-\frac{\eps^{2}}{3}\arrowvert \W_{k}]\geq 0.$$
	as  $\E[\W_{k}\arrowvert\W_{k}]=\W_{k}$ then
	$$\E[\W_{k+1}\arrowvert \W_{k}]\leq\W_{k}.$$
	We have proved that the sequence $\{\W_{k}\}_{k\geq 1}$ is a \textit{supermartingale}. In order to use the \textit{OSTh},
	given a fixed integer $m\in\mathbb{N}$ we define the stopping time
	$$ \tau_{m}=\tau \wedge m := \min\{\tau,m\}.$$
	This new stopping time verifies
	$ \tau_{m}\leq m $
	which is the first hypothesis of the \textit{OSTh}. Then, we obtain
	$$ \E[\W_{\tau_{m}}]\leq \W_{0}. $$
	Observe that $\lim\limits_{m\rightarrow\infty}\tau\wedge m = \tau $ almost surely. Then, using \textit{Fatou's Lemma}, we arrive to
	$$ \E[\W_{\tau}]=\E[\liminf_{m} \W_{\tau\wedge m}]\underbrace{\leq}_{Fatou} \liminf_{m} \E[\W_{\tau\wedge m}]\underbrace{\leq}_{OSTh} \W_{0}.$$
	Thus, we obtain 
	$ \E[\W_{\tau}]\leq \W_{0}$,
	i.e.,
	\begin{equation}
	\label{OST2}
	\E \Big[N^{2}_{\tau}+\frac{\tau\eps^{2}}{3} \Big]\leq N_{0}^{2}.
	\end{equation}
	Then,
	$$\E[\tau]\leq 3(|x_{0}|+\eps^{3})^{2}\eps^{-2}\leq 4|x_{0} |^{2}\eps^{-2}$$
	if $\eps$ is small enough. 
	On the other hand, if we go back to \eqref{OST2} we have
	$$\E[N_{\tau}^{2}]\leq N_{0}^{2},$$
	i.e.
	$$\E[|x_{\tau}|^{2}]\leq\E \Big[(|x_{\tau}|+\frac{\eps^{3}}{2^{\tau}})^{2} \Big]\leq (|x_{0}|+\eps^{3})^{2}\leq 2 |x_{0}|^{2}.$$
	What we have so far is that
	\begin{equation}
	\label{etau}
	\E[\tau]\leq 4|x_{0}|^{2}\eps^{-2}
	\qquad \mbox{and} \qquad \E[|x_{\tau}|^{2}]\leq 2 |x_{0}|^{2}.
	\end{equation}
	Given $\eta > 0$ and $a > 0$, we take $x_{0}\in\Om$ such that $|x_{0} |< r_{0}$ with $r_{0}$ that will be choosed later
	(depending on $\eta$ and $a$). We have
	$$ C r_{0}^{2}\eps^{-2}\geq C |x_{0} |^{2}\eps^{-2} \geq \E^{x_{0}}[\tau]\geq \P \Big(\tau \geq \frac{a}{2\eps^{2}}\Big)\frac{a}{2\eps^{2}}.$$
	Thus
	\begin{equation}
	\label{Ptau}
	\P \Big(\tau \geq \frac{a}{2\eps^{2}} \Big)\leq C 2\frac{r_{0}^{2}}{a}< \eta
	\end{equation}
	which holds true if $r_{0}<\sqrt{\frac{\eta a}{2C}}$.
	
	Now, from \eqref{Ptau} we have
	$$
	\P \Big(\tau\eps^2 < \frac{a}{2} \Big)\geq 1-\eta \quad \Rightarrow \quad \P(t_0-\tau\eps^2 >t_0- \frac{a}{2})\geq 1-\eta.
	$$
	Then, using that $t_{\tau}<t_0$ we obtain
	$$
	\P \Big(|t_{\tau}-t_0|<\frac{a}{2} \Big)\geq 1-\eta.
	$$
	Observe that, if we take $l_0<\frac{a}{2}$, we have
	$$
	|t_{\tau}-s|\leq |t_{\tau}-t_0|+|t_0-s|<|t_{\tau}-t_0|+\frac{a}{2}
	$$
	Then
	$$
	\Big\{|t_{\tau}-t_0|<\frac{a}{2} \Big\}\subseteq \Big\{|t_{\tau}-s|<a \Big\},
	$$
	and we can conclude that 
	$$
	\P(|t_{\tau}-s|<a)\geq 1-\eta.
	$$
	Also we have
	$$ C r_{0}^{2} \geq C |x_{0} |^{2}\geq \E^{x_{0}}[|x_{\tau} |^{2}]\geq a^{2}\P(|x_{\tau} |^{2}\geq a^{2}).$$
	Then 
	$$ \P(|x_{\tau} |\geq a)\leq C\frac{r_{0}^{2}}{a^{2}}< \eta$$
	which holds true if $r_{0}< \sqrt{\frac{\eta a^{2}}{C}}$. Observe that if we take $a<1/2$ we have $\sqrt{\frac{\eta a^{2}}{C}}<\sqrt{\frac{\eta a}{2C}}$, then if we choose $r_{0}<\sqrt{\frac{\eta a^{2}}{C}}$ both conditions are fulfilled at the same time.
	
	\textbf{Case 2:} Suppose that $s=0$, that is we have a point $(y,0)$ with $y\in\overline{\Om}$. One more time, we use as strategy 
	for any of the two players to point towards $y$, as we defined above. that is,
	$$
	x_{k+1}=S^{*}\big(x_0,...,x_k\big)=x_k + \Big(\frac{\eps^3}{2^k}-\eps \Big)\frac{y-x_k}{|y-x_k|}
	$$
	if $|x_k-y|\geq \eps$ and $x_{k+1}=y$ in other case. We can assume without loss of generality that player I uses this strategy. 
	
	Suppose that $0<t_0<l_0$ for some $l_0$ small (to be chosen latter). Then the stopping time is bounded. In fact, 
	$\tau\leq [\frac{l_0}{\eps^2}]$ whit probability 1. Let us call 
	$M=[\frac{l_0}{\eps^2}]$. Since $t_{\tau}< t_0$ it is enough take $l_0\leq a/2$ to get 
	$$
	\P(\tau \geq \frac{a}{2\eps^2})=0 \quad \mbox{ and } \quad \P(|t_{\tau}|<a)=1.
	$$
	Now we define the following random variables:
	$$
	X_{k}=  \left\{
	\begin{array}{ll}
	1 \quad & \quad \mbox{if Player II wins,} \\[10pt]
	-1 \quad & \quad \mbox{if Player I wins,} 
	\end{array}	\right. 
	$$
	for $k\geq 1$ and 
	$$
	Z_k=\sum_{j=1}^{k}X_k.
	$$
	Observe that $X_k$ are independent with $\E[X_k]=0$ and $\mathbb{V}[X_{k}]=1$. 
	Then, $\E[Z_k]=0$ and $\mathbb{V}[Z_k]=k$. If we use Chevichev's Theorem (see \cite{Ferrari}) we obtain
	$$
	\P(|Z_M|\geq \frac{a}{2\eps})\leq \frac{\mathbb{V}[Z_M]}{(\frac{a}{2\eps})^2}=\frac{M4\eps^2}{a^2}\leq 4\frac{(\frac{l_0}{\eps^2}+1)\eps^2}{a^2}\leq\frac{4l_0}{a^2}+\frac{\eps^2}{a^2}<\eta
	$$
	if $\frac{4 l_0}{a^2}<\frac{\eta}{2}$ and $\frac{\eps^2}{a^2}<\frac{\eta}{2}$. This says that the probability that Player II wins $\frac{a}{2\eps}$ more times than Player I is small.
	 Then, we deduce that
	$$
	\P(|x_{\tau}-x_0|\geq\frac{a}{2})<\eta.
	$$
	Here we use that the maximum we can get away from $x_0$ is $\eps$ each step. Now, if we take $r_0<\frac{a}{2}$, we obtain
	$$
	|x_{\tau}-y|\leq |x_{\tau}-x_0|+|x_0-y|<|x_{\tau}-x_0|+\frac{a}{2}.
	$$
	Hence, we have
	$$
	\Big\{|x_{\tau}-y|\geq a \Big\}\subseteq \Big\{|x_{\tau}-x_0|\geq \frac{a}{2}\Big\},
	$$
	and then we conclude that
	$$
	\P({|x_{\tau}-y|\geq a})<\eta.
	$$ 
	This ends the proof.
\end{proof}

\subsection{Estimates for the Random Walk game}

Here we assume that we are playing on board 2, with the random walk game without changing time.
The estimates for this game follow the same ideas as before, and are even simpler since there are no strategies 
of the players involved in this case. We include some ideas for completeness and refer to
\cite{nosotros} for more details. Recall that in this board the time $t$ does not change
when we play.

\begin{Lemma}
	Given $\eta>0$ and $a>0$, there exists $r_{0}>0$ and $\eps_{0}>0$ such that, given $(y,s)\in\big[\partial\Om\times (0,T)\big]$ and $x_{0}\in\Om$ with $|x_{0}-y|<r_{0}$, if we play random in $\Om\times \{ s \}$ we obtain 
	$$\mathbb{P} \Big(|x_{\tau}-y|< a \Big) \geq 1 - \eta
	\qquad \mbox{and} \qquad\mathbb{P} \Big(\tau \geq \frac{a}{2\eps^2} \Big)< \eta$$
	for $\eps<\eps_{0}$ and $(x_{\tau},s)$ the first position outside $\Om\times (0,T]$.
\end{Lemma}

\begin{proof}
We include only a skecht of the proof. Extra details can be found in \cite{nosotros}. Assume that $ N \geq 3 $
(the cases $N=1,2$ are similar).
The first step is, given $ \theta <\theta_{0} $, and $ y \in \Om $ we are going to assume that $ z_{y} = 0 $ is chased so that we have 
$ \ol {B_{\theta}(0)}  \cap \ol {\Om} = \{ y \}$. We define the set $\Om_{\eps}=\{x\in\R^{N}:d(x,\Om)<\eps\}$
for $ \eps $ small enough. Now, we consider the function $\mu:\Om_{\eps}\rightarrow\R$ given by 
\begin{equation}
\label{mu}
\mu(x)=\frac{1}{\theta^{N-2}}-\frac{1}{|x |^{N-2}}.
\end{equation}
This function is positive in $ \ol{\Om}\backslash \{ y \} $,
radially increasing and harmonic in $ \Om $. Also it holds that $ \mu (y) = 0 $.

Take the first position of the game, $ x_{0} \in \Om $, such that $ |x_{0} -y |<r_{0} $ with $ r_{0} $ 
to be choosen later. Let $ (x_{k})_{k \geq 0} $ be the sequence of positions of the game playing random walks
and
consider the sequence of random variables
$$ N_{k}=\mu(x_{k})$$
for $ k \geq 0 $. Since $ \mu $ is harmonic, we have that $ N_{k} $ is a martingale,
$$ \E [N_{k + 1} \arrowvert N_{k}] = \kint_{B _{\eps} (x_{k})} \mu(y) dy = \mu (x_{k}) = N_{k}. $$
Since $ \mu $ is bounded in $\Omega$, 
the third hypothesis of {\it OSTh} is fulfilled, hence we obtain
\begin{equation}
\label{muxo}
\E [\mu(x_{\tau})] = \mu(x_ {0}).
\end{equation}
We have the following estimate for $\mu(x_{0})$: there exists a constant $c(\Om,\theta)>0$ such that
$$
\mu(x_0)\leq c(\Om,\theta)r_0.
$$
Now we need to establish a relation between $ \mu (x_{\tau}) $ and $ |x_{\tau} -y |$. 
To this end, we take the function $ b: [\theta, + \infty) \rightarrow \R $ given by
\begin{equation}
\label{funb}
b(\ol{a})=\frac{1}{\theta^{N-2}}-\frac{1}{\ol{a}^{N-2}}.
\end{equation}
Note that this function is the radial version of $ \mu $. It is positive and increasing,
then, it has an inverse (also increasing) that is given by the formula
$$ \ol{a}(b)=\frac{\theta}{(1-\theta^{N-2}b)^{\frac{1}{N-2}}}.$$
With this function we can get the following result:
given $a>0$, exist $\ol{a}>\theta$, $b>0$ and $\eps_{0}>0$ such that  
$$
\mbox{if } \mu(x_{\tau})<b \Rightarrow |x_{\tau}-y|<a \ , \ d(x_{\tau},\Om)<\eps_{0}.
$$
Then, we have
$$\P(\mu(x_{\tau})\geq b)\geq \P(|x_{\tau}-y|\geq a).$$
and then we obtain  
\begin{equation}
\label{desnorma}
\P(|x_{\tau}-y|\geq a)<\eta
\end{equation}
if $r_0$ is small.
Now, let us compute
\begin{equation}
\label{ENK}
\E[N_{k+1}^{2}-N_{k}^{2}\arrowvert N_{k}]=\kint_{B_{\eps}(x_{k})}(\mu^{2}(w)-\mu^{2}(x_{k}))dw.
\end{equation}
Using the Taylor expansion of order two we can prove that
$$
 \E[N_{k+1}^{2}-N_{k}^{2}\arrowvert N_{k}]\geq \sigma(\Om)\eps^{2}.
 $$
 Then, with arguments similar to those used above we obtain
$$
\P \Big(\tau\geq\frac{a}{2\eps^{2}}\Big)<\eta
$$
for $r_{0}$ small enough.
\end{proof} 

Now we are ready to prove the second condition in the Arzela-Ascoli type lemma.

\begin{Lemma}\label{lem.ascoli.arzela.asymp} Given $\delta>0$ there are 
	$r_0>0$ and $\eps_0>0$ such that for every $0< \eps < \eps_0$
	and any $x, y \in \overline{\Omega}$ with $|x - y | < r_0 $ and $|t-s|< r_0$
	it holds
	$$
	|u^\eps (x,t) - u^\eps (y,s)| < \delta \qquad \mbox{and} \qquad |v^\eps (x,t) - v^\eps (y,s)| < \delta.
	$$
\end{Lemma}

\begin{proof} We deal with the estimate for $u^\eps$. 
	Recall that $u^{\eps}$ is the value of the game playing in the first board (where we play Tug-of-War).

	First, we start with two close points $(x,t)$ and $(y,s)$ with $(y,s)\in \partial_p \Omega_T$ and $(x,t)\in \Omega_T$. 
Let define the following function $w:[(\R^N\backslash\Om\times (0,T])\cup(\R^N\times\{0\})]\rightarrow\R$,
$$
w(x,t) =\left\lbrace
\begin{array}{ll}
\ol{f}(x,t) &  \ \ \mbox{if} \quad t\geq 0, x\notin\Om,  \\[10pt]
\displaystyle u_0(x) &  \ \ \mbox{if} \quad t=0, x\in\Om.
\end{array}
\right.
$$
From our conditions on the data, the function $w$ is well defined and is Lipschitz in both variables, that is $$|w(x,t)-w(y,s)|\leq L(|x-y|+|t-s|).$$ 
For instance, notice that for $t>0$ and $s=0$ we have 
$$
|\ol{f}(x,t)-u_0(y)|\leq L(|x-y|+|t|).
$$
	Given $\eta >0$ and $a>0$, we have $r_{0}$, $l_0$, $\eps_{0}$ and $S^{*}_{I}$ the strategy as in Lemma \ref{lema.estim.ToW}.
	Let 
	$$ F=\Big\{\mbox{the position does not change board in the first } \ \lceil \frac{a}{2\eps^{2}}\rceil \mbox{ plays and } \tau < \lceil \frac{a}{2\eps^{2}}\rceil \Big\}.$$
	
	We consider two cases.
	
	\textbf{1st case:} We are going to show that $u^{\eps}(x_{0},t_0)-w(y,s) \geq - A(a,\eta)$ with $A(a,\eta)\searrow 0$ if $a\rightarrow 0$ and 
	$\eta\rightarrow 0$.
	We have
	$$u^{\eps}(x_{0},t_0)\geq \inf_{S_{II}}\E^{(x_{0},t_0}_{S^{*}_{I},S_{II}}[h(x_{\tau},t_{\tau})].$$
	Now, using probality properties we obtain
	$$
	\begin{array}{l}
	\displaystyle 
	\E^{(x_{0},t_0)}_{S^{*}_{I},S_{II}}[ h (x_{\tau},t_{\tau})]  = \E^{(x_{0},t_0)}_{S^{*}_{I},S_{II}}[h(x_{\tau},t_{\tau})\arrowvert F]\P(F)
	+  \E^{(x_{0},t_0)}_{S^{*}_{I},S_{II}}[h(x_{\tau},t_{\tau})\arrowvert F^{c}]\P(F^{c})
	\\[10pt]
	\qquad \displaystyle \geq \E^{(x_{0},t_0)}_{S^{*}_{I},S_{II}}[w(x_{\tau},t_{\tau})\arrowvert F]\P(F)-\max\{\lvert\ol{f}|,\lvert\ol{g}|,|u_0|\}\P(F^{c}).
	\end{array}
	$$
	Now we estimate $\P(F)$ and $\P(F^{c})$. We have that  
	$$\P(F^{c})\leq \P \Big(\mbox{the game changes board before }\lceil \frac{a}{2\eps^{2}}\rceil \mbox{ plays}\Big)
	+\P(\tau\geq\lceil \frac{a}{2\eps^{2}}\rceil).$$
	Hence we are left with two bounds. First, we have 
	\begin{equation}
	\label{Ac1}
	\P \Big(\mbox{the game changes board before }\lceil \frac{a}{2\eps^{2}}\rceil \mbox{ plays} \Big)=1-(1-\eps^{2})^{\frac{a}{2\eps^{2}}} \leq (1-e^{-a/2})+\eta
	\end{equation}
	for $\eps$ small enough.
	Here we are using that $(1-\eps^{2})^{\frac{a}{2\eps^{2}}}\nearrow e^{-a/2}$.
	Next, we observe that using Lemma \ref{lema.estim.ToW} we get
	\begin{equation}
	\label{Ac2}
	\P\Big(\tau \geq \frac{a}{2\eps^{2}}\Big) \leq \P\Big(\tau \geq \frac{a}{2\eps_{0}^{2}}\Big)< \eta,
	\end{equation}
	for $\eps < \eps_{0}$.
	From \eqref{Ac1} and \eqref{Ac2} we obtain
	$$ \P(F^{c})\leq (1-e^{-a/2})+\eta +\eta= (1-e^{-a/2})+2\eta $$
	and hence
	$$ \P(F) =1-\P(F^{c}) \geq 1-[(1-e^{-a/2})+2\eta] .$$
	Then we obtain
	\begin{equation}
	\label{arriba}
	\begin{array}{l}
	\displaystyle  \E^{(x_{0},t_0)}_{S^{*}_{I},S_{II}}[h(x_{\tau},t_{\tau})] 
	\displaystyle \geq \E^{(x_{0},t_0)}_{S^{*}_{I},S_{II}}[w(x_{\tau},t_{\tau})\arrowvert F] (1-[(1-e^{-a/2})+2\eta])
	\\[10pt]
	\qquad \qquad  \qquad  \qquad  \qquad \qquad \displaystyle -\max\{\lvert\ol{f}|,\lvert\ol{g}|,|u_0|\}[(1-e^{-a/2})+2\eta] .
	\end{array}
	\end{equation}
	Let us analyze the expected value $\E^{(x_{0},t_0)}_{S^{*}_{I},S_{II}}[w(x_{\tau},t_{\tau})\arrowvert F]$. 
	Again we need to consider two events,
	$$ F_{1}=F\cap [\{ |x_{\tau}-y|< a\}\cap\{ |t_{\tau}-s|<a\}] \qquad \mbox{and} \qquad F_{2}=F\cap F_{1}^{c}.$$
	We have that $ F=F_{1}\cup F_{2}$.
	Then 
	\begin{equation}
	\label{retomo}
	\E^{(x_{0},t_0)}_{S^{*}_{I},S_{II}}[w(x_{\tau},t_{\tau})\arrowvert F]=\E^{(x_{0},t_0)}_{S^{*}_{I},S_{II}}[w(x_{\tau},t_{\tau})\arrowvert F_{1}]\P(F_{1})+\E^{(x_{0},t_0)}_{S^{*}_{I},S_{II}}[w(x_{\tau},t_{\tau})\arrowvert F_{2}]\P(F_{2}).
	\end{equation}
	Now we have that
	\begin{equation}
	\label{a2}
	\begin{array}{l}
	\displaystyle  \P(F_{2})\leq \P([\{ |x_{\tau}-y|< a\}\cap\{ |t_{\tau}-s|<a\}]^c)= P(\{ |x_{\tau}-y|\geq a\}\cup \{ |t_{\tau}-s|\geq a\}) \\[10pt]
	\qquad  \qquad \leq \P(|x_{\tau}-y|\geq a)+\P(|t_{\tau}-s|\geq a)<2\eta.
	\displaystyle  
	\end{array}
	\end{equation}
	To get a bound for the other case we observe that
	$ F_{1}^{c}=F^{c}\cup  \{ |x_{\tau}-y|\geq a\}\cup \{|t_{\tau}-s|\geq a\}$.
	Therefore 
	$$ \P(F_{1})=1-\P(F_{1}^{c})\geq 1-[\P(F^{c})+\P(|x_{\tau}-y|\geq a)+\P(|t_{\tau}-s|\geq a)],$$
	and we arrive to
	\begin{equation}
	\label{a1}
	\P(F_{1})\geq 1-[(1-e^{-a})+2\eta+\eta+\eta]=1-[(1-e^{-a})+4\eta].
	\end{equation}
	If we go back to \eqref{retomo} and use \eqref{a1} and \eqref{a2} we get 
	\begin{equation}
	\label{retomo2}
	\E^{x_{0}}_{S^{*}_{I},S_{II}}[w(x_{\tau},t_{\tau})\arrowvert F]\geq\E^{x_{0}}_{S^{*}_{I},S_{II}}[w(x_{\tau},t_{\tau})\arrowvert F_{1}](1-[(1-e^{-a})+4\eta])-\max \{|\ol{f}|,|u_0|\}2\eta .
	\end{equation} 
	Using that $w$ is Lipschitz we obtain 
	$$ w(x_{\tau},t_{\tau})\geq w(y,s)-L(|x_{\tau}-y|+|t_{\tau}-s|)\geq w(y,s)-2La ,$$
	and then, using that $(w(y,s)-2La)$ does not depend on the strategies, we conclude that 
	\begin{equation}
	\E^{x_{0}}_{S^{*}_{I},S_{II}}[w(x_{\tau},t_{\tau})\arrowvert F]\geq(w(y,s)-2La)(1-[(1-e^{-a})+4\eta])-\max \{\lvert\ol{f}|,|u_0|\}2\eta.
	\end{equation}
	Recalling \eqref{arriba} we obtain
	$$
	\begin{array}{l}
	\displaystyle  \E^{x_{0}}_{S^{*}_{I},S_{II}}[h ( x_{\tau},t_{\tau})] \displaystyle \\[10pt]
	\quad  \geq ((w(y,s)-2La)(1-[(1-e^{-a})+4\eta]) -\max \{\lvert\ol{f}|,|u_0|\}2\eta ) (1-[(1-e^{-a})+2\eta]) 
	\displaystyle  \\[10pt] \qquad  -\max\{\lvert\ol{f}|,\lvert\ol{g}|,|u_0|\}[(1-e^{-a})+2\eta].
	\end{array}
	$$ 
	Notice that when $\eta \rightarrow 0$ and $a\rightarrow 0$ the the right hand side
	goes to $w(y,s)$, hence we have obtained
	$$ \E^{x_{0}}_{S^{*}_{I},S_{II}}[h( x_{\tau},t_{\tau})]\geq w(y,s)- A(a,\eta)$$
    Taking the infimum over all possible strategies $S_{II}$ and then supremum over $S_{I}$ we get 
	$$ u^{\eps}(x_{0},t_0)\geq w(y,s)- A(a,\eta)$$ 
	with $A(a,\eta)\to 0$ as $\eta\rightarrow 0$ and $a\rightarrow 0$ as we wanted to show.
	
	\textbf{2nd case:} Now we want to show that $$u^{\eps}(x_{0},t_0)-w(y,s)\leq B(a,\eta)$$ with $B(a,\eta)\searrow 0$ as $\eta\rightarrow 0$ and 
	$a\rightarrow 0$.
	In this case we just use the strategy $S^*$ from Lemma \ref{lema.estim.ToW} as the strategy for the second player
	$S^{*}_{II}$ and we obtain
	$$ u^{\eps}(x_{0},t_0)\leq \sup_{S_{II}}\E^{x_{0}}_{S_{I},S^{*}_{II}}[h(x_{\tau},t_{\tau})].$$
	Using again the set $F$ that we considered in the previous case we obtain
	$$ \E^{x_{0}}_{S_{I},S^{*}_{II}}[ h( x_{\tau},t_{\tau})]= \E^{x_{0}}_{S_{I},S^{*}_{II}}[w( x_{\tau},t_{\tau})\arrowvert F]\P(F)+ 
	\E^{x_{0}}_{S_{I},S^{*}_{II}}[h( x_{\tau},t_{\tau})\arrowvert F^{c}]\P(F^{c}).$$
	We have that $\P(F) \leq 1$ and $\P(F^{c})\leq (1-e^{-a/2})+2\eta $. Hence we get
	\begin{equation}
	\label{retomo3}
	\E^{x_{0}}_{S_{I},S^{*}_{II}}[h( x_{\tau},t_{\tau})]\leq \E^{x_{0}}_{S_{I},S^{*}_{II}}[w( x_{\tau},t_{\tau})\arrowvert F]+\max\{\lvert\ol{f}|,\lvert\ol{g}|,|u_0|\}[(1-e^{-a/2})+2\eta].
	\end{equation}
	To bound $\E^{x_{0}}_{S_{I},S^{*}_{II}}[w(x_{\tau},t_{\tau})\arrowvert F]$ we will use again the sets $F_{1}$ and $F_{2}$ 
	as in the previous case. We have
	$$ \E^{x_{0}}_{S_{I},S^{*}_{II}}[w(x_{\tau},t_{\tau})\arrowvert F]=\E^{x_{0}}_{S_{I},S^{*}_{II}}[w(x_{\tau},t_{\tau})\arrowvert F_{1}]\P(F_{1})+\E^{(x_{0},t_0)}_{S_{I},S^{*}_{II}}[w(x_{\tau},t_{\tau})\arrowvert F_{2}]\P(F_{2}).$$
	Now we use that $\P(F_{1}) \leq 1$ and $\P(F_{2})\leq c\eta$ to obtain
	$$ \E^{(x_{0},t_0)}_{S_{I},S^{*}_{II}}[w(x_{\tau},t_{\tau})\arrowvert F]\leq \E^{x_{0}}_{S_{I},S^{*}_{II}}[w(x_{\tau},t_{\tau})\arrowvert F_{1}]+\max\{\lvert\ol{f}|,|u_0|\}2\eta .$$
    Using that $w$ is Lipschitz and that $(w(y,s)+2La)$ does not depend on the strategies we get
	$$ \E^{x_{0}}_{S_{I},S^{*}_{II}}[w(x_{\tau},t_{\tau})\arrowvert F]\leq \E^{x_{0}}_{S_{I},S^{*}_{II}}[w(y,s)+2La\arrowvert F_{1}]
	+\max\{\lvert\ol{f}|\}2\eta \leq(w(y,s)+2La)+\max\{\lvert\ol{f}|\}2\eta ,$$
	and therefore we conclude that 
	$$ \E^{x_{0}}_{S_{I},S^{*}_{II}}[h( x_{\tau},t_{\tau} )]\leq w(y,s)+2La+\max\{\lvert\ol{f}|\}2\eta + \max\{\lvert\ol{f}|,\lvert\ol{g}|\}[(1-e^{-a/2})+2\eta].
	$$
	We have proved that 
	$$ \E^{x_{0}}_{S_{I},S^{*}_{II}}[ h( x_{\tau},t_{\tau})]\leq w(y,s) + B(a,\eta) $$
	with $B(a,\eta)\to 0$. Taking supremum over the strategies for Player I we obtain 
	$$ u^{\eps}(x_{0},t_0)\leq w(y,s)+B(a,\eta)$$
	with $B(a,\eta)\rightarrow 0$ as $\eta\rightarrow 0$ and $a\rightarrow 0$. 
	
We conclude that 
	$$ |u^{\eps}(x_{0},t_0)-w(y,s)|< \max\{A(a,\eta),B(a,\eta)\},$$
	that holds when $(y,s)\in\partial_p \Om_T$ and $(x_0,t_0)\in\Om_T$ is close to $(y,s)$.
	
	An analogous estimate holds for $v^\eps$. The details are simpler and left to the reader.  
	
	Now, given two points $(x_0,t_0)$ and $(z_0,s_0)$ inside $\Omega$ with $|x_0-z_0|<r_0$ and $|t_0-s_0|<r_0$ we couple the  
	game starting at $(x_0,t_0)$ with the game starting at $(z_0,s_0)$ making the same movements and also
	changing board simultaneously. This coupling generates two sequences of positions $(x_i,t_i,j_i)$ and $(z_i,s_i,k_i)$
	such that $|x_i - z_i|<r_0$, $|t_i - s_i|<l_0$ and $j_i=k_i$ (since they change boars at the same time both games are at
	the same board at every turn). This continues until one of the games exits the domain (say at $(x_\tau,t_{\tau}) \not\in \Omega\times (0,T)$).
	At this point for the game starting at $(z_0,s_0)$ we have that its position $(z_\tau,s_{\tau})$ is close to the exterior point $(x_\tau,t_{\tau}) \not\in \Omega\times (0,T)$ (since we
	have $|x_\tau - z_\tau|<r_0$ and $|t_\tau - s_\tau|<r_0$) and hence we can use our previous estimates for points close to the boundary to conclude that 
	$$ |u^{\eps}(x_{0},t_0)- u^\eps (z_0,s_0)|< \delta, \qquad \mbox{ and }
	\qquad |v^{\eps}(x_{0},t_0)- v^\eps (z_0,s_0)|< \delta. $$
	This ends the proof.
\end{proof}

As a consequence, we have convergence of $(u^{\eps},v^{\eps})$ as $\eps \to 0$ along subsequences. 

\begin{theorem} \label{teo.conv.unif}
	Let $(u^{\eps},v^{\eps})$ be solutions to the DPP, then there exists a subsequence
	$\eps_k \to 0$ and a pair on functions $(u,v)$ continuous in $\overline{\Omega}$ such that  
	$$
	u^{\eps_k} \to u, \qquad \mbox{ and } \qquad  v^{\eps_k} \to v, 
	$$ 
	uniformly in  $\overline{\Omega}\times [0,T]$.
\end{theorem}

\begin{proof} Lemma \ref{lem.ascoli.arzela.acot} and
	Lemma \ref{lem.ascoli.arzela.asymp} imply that we can use 
	the Arzela-Ascoli type lemma, Lemma \ref{lem.ascoli.arzela}, to obtain uniform convergence. 
\end{proof}

\section{Existence of viscosity solutions to \eqref{ED1}}
\label{sect-limite.viscoso}

In this section we use viscosity arguments to show that a uniform limit of the values of the game (which exists due to
Theorem \ref{teo.conv.unif}) is in fact a viscosity solution to our parabolic/elliptic system \eqref{ED1}.

\textit{First equation.} Let be $\phi\in C^{2}(\Om \times (0,T])$ such that $(u-\phi)(x_{0},t_0)$ has a absolute maximum at $(x_{0},t_0)$
 with $(u-\phi)(x_{0},t_0)=0$ (maximum in the two variables $x$ and $t$). 
 Then exist a sequence $(x_{\eps}, t_\eps)_{\eps>0}$ such that $x_{\eps}\rightarrow x_0$ and 
 $t_{\eps}\rightarrow t_0$ when $\eps\rightarrow 0$ and 
$$
u^{\eps}(y,t)-\phi(y,t)\leq u^{\eps}(x_{\eps},t_\eps)-\phi(x_{\eps},t_\eps)+\eps^{3}
$$
if $y\in\ol{\Om}$, $t\in (0,T]$. Then we obtain 
\begin{equation}
\label{ast}
 u^{\eps}(y,t)-u^{\eps}(x_{\eps},t_\eps) \leq \phi(y,t)-\phi(x_{\eps},t_\eps)+\eps^{3}
\end{equation}
Then, using the DPP for $u^\eps$ at the point $(x_\eps,t_\eps)$ we have 
$$
\displaystyle u^{\eps}(x_\eps,t_\eps)=\eps^{2}v^{\eps}(x_\eps,t_\eps)+(1-\eps^{2})\Big\{\half \sup_{y \in B_{\eps}(x_\eps)}u^{\eps}(y,t_\eps-\eps^{2}) + \half \inf_{y \in B_{\eps}(x_\eps)}u^{\eps}(y,t_\eps-\eps^{2})\Big\} 
$$
and then, adding and subtracting $u^{\eps}(x_\eps,t_\eps -\eps^2)$, we obtain
 $$
 \begin{array}{l}
\displaystyle 0=\eps^{2}v^{\eps}(x_\eps,t_\eps) -  \Big[u^{\eps}(x_\eps,t_\eps) - u^{\eps}(x_\eps,t_\eps -\eps^2)\Big]
+(1-\eps^{2})\half 
\sup_{y \in B_{\eps}(x_\eps)}\Big[u^{\eps}(y,t_\eps-\eps^{2}) - u^{\eps}(x_\eps,t_\eps-\eps^2)\Big]  \\[10pt]
\qquad \displaystyle + (1-\eps^{2}) \half \inf_{y \in B_{\eps}(x_\eps)}
\Big[u^{\eps}(y,t_\eps-\eps^{2}) - u^{\eps}(x_\eps,t_\eps-\eps^2)\Big] - \eps^2 u^{\eps}(x_\eps,t_\eps-\eps^2).
\end{array}
$$
Now, using \eqref{ast} we arrive to
 $$
 \begin{array}{l}
\displaystyle 0\leq \eps^{2}v^{\eps}(x_\eps,t_\eps) -  \Big[\phi(x_\eps,t_\eps) - \phi(x_\eps,t_\eps -\eps^2)\Big]
+(1-\eps^{2})\half 
\sup_{y \in B_{\eps}(x_\eps)}\Big[\phi(y,t_\eps-\eps^{2}) - \phi(x_\eps,t_\eps-\eps^2)\Big]  \\[10pt]
\qquad \displaystyle + (1-\eps^{2}) \half \inf_{y \in B_{\eps}(x_\eps)}
\Big[\phi (y,t_\eps-\eps^{2}) - \phi (x_\eps,t_\eps-\eps^2)\Big] + \eps^3
 - \eps^2 u^{\eps}(x_\eps,t_\eps-\eps^2).
\end{array}
$$
From this point the proof follows from simple Taylor expansions.
  If $\nabla\phi(x_{0},t_0)\neq 0$ then  $\nabla\phi(x_{\eps},t_\eps-\eps^{2})\neq 0$ for $\eps$ small enough. Let us call 
    $$
    w_{\eps}= \frac{\nabla\phi(x_{\eps},t_\eps-\eps^{2})}{\lvert \nabla\phi(x_{\eps},t_\eps-\eps^{2})\rvert} \qquad and \qquad  w_0 = \frac{\nabla\phi(x_0,t_0)}{\lvert \nabla\phi(x_0,t_0)\rvert}
  $$
  Observe that $w_\eps \rightarrow w_0$ when $\eps\rightarrow 0$. Then 
  $$ \sup_{y\in B_{\eps}(x_{\eps})}\phi(y,t_\eps-\eps^{2})\sim\phi(x_{\eps}+\eps w_{\eps},t_\eps-\eps^{2})$$
   and 
   $$
   \inf_{y\in B_{\eps}(x_{\eps})}\phi(y,t_\eps-\eps^{2})\sim\phi(x_{\eps}-\eps w_{\eps},t_\eps-\eps^{2}).
  $$
  Thus
 $$
  \begin{array}{l}
 \displaystyle 0\leq \eps^{2}v^{\eps}(x_\eps,t_\eps) -  \Big[\phi(x_\eps,t_\eps) - \phi(x_\eps,t_\eps -\eps^2)\Big]
 +(1-\eps^{2})\half 
 \Big[\phi(x_{\eps}+\eps w_{\eps},t_\eps-\eps^{2})- \phi (x_\eps,t_\eps-\eps^2) \Big]  \\[10pt]
 \qquad \displaystyle + (1-\eps^{2}) \half 
 \Big[ \phi(x_{\eps}-\eps w_{\eps},t_\eps-\eps^{2})- \phi (x_\eps,t_\eps-\eps^2)  \Big] + \eps^3
 - \eps^2 u^{\eps}(x_\eps,t_\eps-\eps^2).
 \end{array}
 $$
 
Using the Taylor expansion of $\phi$ with respect to the spatial variables at the point $(x_{\eps},t_\eps-\eps^2)$ we obtain
$$ \phi(x_{\eps}+\eps w_{\eps},t_\eps-\eps^2)-\phi(x_{\eps},t_\eps-\eps^2) =\eps \langle \nabla \phi(x_{\eps},t_\eps-\eps^2),w_{\eps}\rangle+\eps^{2} \half \langle D^{2}\phi(x_{\eps},t_\eps-\eps^2)w_{\eps},w_{\eps}\rangle + \textit{o}(\eps^{2})$$
$$ \phi(x_{\eps}-\eps w_{\eps},t_\eps-\eps^2)-\phi(x_{\eps},t_\eps-\eps^2)= -\eps \langle \nabla \phi(x_{\eps},t_\eps-\eps^2),w_{\eps}\rangle+\eps^{2} \half \langle D^{2}\phi(x_{\eps},t_\eps-\eps^2)w_{\eps},w_{\eps}\rangle+\textit{o}(\eps^{2}).$$
So, dividing by $\eps^{2}$, we get
 $$
    \begin{array}{l}
   \displaystyle 0\leq v^{\eps}(x_\eps,t_\eps) - \frac{\phi(x_\eps,t_\eps) - \phi(x_\eps,t_\eps -\eps^2)}{\eps^2}
   +(1-\eps^{2})\half \langle D^{2}\phi(x_{\eps},t_\eps-\eps^2)w_{\eps},w_{\eps}\rangle +\frac{\textit{o}(\eps^{2})}{\eps^2}  
    - u^{\eps}(x_\eps,t_\eps-\eps^2).
   \end{array}
   $$
Taking limit $\eps\rightarrow 0$ we conclude
$$ 0\leq v (x_0,t_0) -
\frac{\partial \phi}{\partial t} (x_0,t_0) +\displaystyle \half \langle D^{2}\phi(x_0,t_0)w_0,w_0\rangle  - u (x_0,t_0) 
$$
i.e.
$$
\frac{\partial \phi}{\partial t} (x_0,t_0) - \displaystyle \half \Delta_{\infty}\phi (x_0,t_0)+u (x_0,t_0)- v (x_0,t_0)\leq 0,
$$
as we wanted to show. 

The reverse inequality for smooth test functions that touches from below the graph of $u$  can be obtained with analogous arguments.

\textit{Second equation.}
Now, let us show that $v$ is a viscosity solution to 
$$ -\frac{\kappa}{2}\Delta v(x,t)+v(x,t)-u(x,t)=0.$$

Let us start by showing that $v$ is a subsolution. For a fix $t>0$, let $\psi \in C^{2}(\Om)$ such that 
$v(x_{0},t)-\psi(x_{0})=0$ and has a maximum of $v(\cdot,t)-\psi$ at $x_{0} \in \Om$. As before, we have the existence of a sequence $(x_{\eps})_{\eps>0}$ such that 
$x_{\eps} \rightarrow x_{0}$ and 
\begin{equation}
\label{ast44}
 v^{\eps}(y,t)-v^{\eps}(x_{\eps},t) \leq \psi(y)-\psi(x_{\eps})+\eps^{3}.
\end{equation}
Therefore, from the DPP \eqref{DPP} at the point $(x_\eps,t)$, we obtain
$$
0= (u^{\eps}(x_{\eps},t)-v^{\eps}(x_{\eps},t))+(1-\eps^{2})\frac{1}{\eps^{2}}\kint_{B_{\eps}(x_{\eps})}(v^{\eps}(y,t)-v^{\eps}(x_{\eps},t))dy.
$$
Using \eqref{ast44} we get
$$
0\leq (u^{\eps}(x_{\eps},t)-v^{\eps}(x_{\eps},t))+(1-\eps^{2})\frac{1}{\eps^{2}}\kint_{B_{\eps}(x_{\eps})}(\psi(y)-\psi(x_{\eps}))dy.
$$

From Taylor's expansions we obtain
$$ \frac{1}{\eps^{2}}\kint_{B_{\eps}(x_{\eps})}(\psi(y)-\psi(x_{\eps}))dy=\frac{\kappa}{2} \sum\limits_{j=1}^{n} \partial_{x_{j}x_{j}}\psi(x_{\eps})=\frac{\kappa}{2}\Delta \psi(x_{\eps}),$$
with 
$$\kappa = \frac{1}{\eps^{n}\lvert B_{1}(0)\rvert}\int_{B_{1}(0)}z_{j}^{2}\eps^{n}dz=\frac{1}{\lvert B_{1}(0)\rvert}\int_{B_{1}(0)}z_{j}^{2}dz. $$
Taking limits as $\eps \rightarrow 0$ we get
\begin{equation}
-\frac{\kappa}{2}\Delta \psi(x_{0})+v(x_{0},t)-u(x_{0},t) \leq 0.
\end{equation}

The fact that $v$ is a supersolution is similar.

\subsection{Comparison principle and uniqueness for the limit problem}

Our goal is to show uniqueness for viscosity solutions to our system \eqref{ED1}.
To this end we follow ideas from \cite{BB,nosotros,Mitake} (see also \cite{Jen} for uniqueness results concerning the
infinity Laplacian). This uniqueness result implies that
the whole sequence $u^\eps,v^\eps$ converge as $\eps \to 0$. As usual in viscosity theory, uniqueness
follows from a comparison principle. The comparison principle for the elliptic counterpart
of \eqref{ED1} was obtained in \cite{nosotros}. 

\begin{theorem} \label{teo-compar} Assume that $(u_1, v_1)$ and $(u_2, v_2)$ are a bounded viscosity subsolution 
and a bounded viscosity supersolution of \eqref{ED1}, respectively, and also assume that $u_1 \leq u_2$ and $v_1 \leq v_2$ on 
$\partial_p \Omega_T$. Then, 
$$
u_1 \leq u_2 \qquad \mbox{and} \qquad v_1 \leq v_2,
$$
in $\Omega \times (0,T)$.
\end{theorem}

As an immediate corollary of this comparison result we obtain the desired uniqueness for \eqref{ED1}. 

\begin{corol} \label{corl-unicidad}
There exists a unique viscosity solution to \eqref{ED1}.
\end{corol}

\begin{proof}[Proof of Theorem \ref{teo-compar}]
Following the classical ideas given in \cite{CIL} we perturb the subsolution defining the function
 $$\widetilde{u}_1(x,t):=u_1(x,t)-\frac{\varepsilon}{T-t},\quad\varepsilon>0,$$
  that satisfies, in the viscosity sense,
  $$\begin{cases}
			\displaystyle \frac{\partial (\widetilde{u}_1)}{\partial t} (x,t)- \frac12 \Delta_\infty \widetilde{u}_1 (x,t)
			+ \widetilde{u}_1 (x,t) - v_1(x,t)
			\leq -\frac{\delta}{T-t}
			-\frac{\delta}{(T-t)^2}\leq -c <0 &\text{in }\Omega\times(0,T),\\[10pt]
			\widetilde{u}_1\to -\infty, &\text{when }
				t\to T^{-}.\\
				\end{cases}
				$$
It is clear that if we prove $\widetilde{u}_1(x,t)\leq v_1(x,t)$ for $x\in\Omega$ and $t\in (0,T)$ then the conclusion follows by taking $\delta\to 0$.  Therefore, we can assume that we have a strict subsolution,
   $$\begin{cases}
			\displaystyle \frac{\partial ({u}_1)}{\partial t} (x,t)- \frac12 \Delta_\infty {u}_1 (x,t)
			+ {u}_1 (x,t) - v_1(x,t)
			\leq -c <0 &\text{in }\Omega\times(0,T),\\[10pt]
			{u}_1\to -\infty, &\text{when }
				t\to T^{-}.
				\end{cases}
				$$
				
 Now, suppose, by contradiction, that
 $$\sup_{x\in\overline{\Omega},\, t\in [0,T]} \max \Big\{u_1(x,t)-u_2(x,t); v_1 (x,t) -v_2(x,t) \Big\}:=\eta>0.$$

We may assume further that $u_1$ and $v_1$ are semi-convex in space
and $u_2$ and $v_2$ are semi-concave in space by using sup and inf convolutions and restricting the problem to a slightly smaller domain 
if necessary (see \cite{Mitake} for extra details). 
We now perturb $u_1$ and $v_1$. For $\alpha > 0$, take $\Omega_\alpha := \{x \in \Omega : 
dist(x,\partial \Omega) > \alpha\}$ and for $|h|$ sufficiently small we consider
$$
\begin{array}{l}
\displaystyle 
M(h):=\max \Big\{ \max_{x \in \Omega,\, t,r \in [0,T]} u_1(x+h,t)- u_2(x,r) -\frac{|t-r|^2}{2\varepsilon}; \max_{x \in \Omega,\, t,r \in [0,T]} (v_1(x+h,t)- v_2(x,r))\Big\}  \\[10pt]
\qquad \qquad \displaystyle = w_1(x_h+h,t_h)- w_2(x_h,r_h) 
\end{array}
 $$
for $w=u \mbox{ or } v$ (we will call $w$ the component at which the maximum is achieved)
 and some $x_h \in \Omega_{|h|}$, $t_h,r_r$ (notice that these quantities also depend on $\varepsilon$ 
 but we omit this dependence to simplify the notation). 
 Since $M(0) > 0$, for $|h|$ small enough, we have $M(h)>0$ and the above maximum is the same if we take it over $\Omega_\alpha$ any $\alpha >0$ sufficiently small and fixed.

 As in \cite{nosotros} we obtain that 
there exists a sequence $h_n \to 0$ such that at any maximum point $y \in \Omega_{|h_n|}$ 
 of $$\max \Big\{ \max_{x\in \Omega_{|h_n|}} (u_1(x + h_n) - u_2(x)) ;  \max_{x\in \Omega_{|h_n|}} (u_1(x + h_n) - u_2(x))\Big\},$$ 
 we have $$Dw_1(y + h_n) =
Dw_2(y) \neq 0$$ for $ n \in \mathbb{N}$.
 
 Now we consider, as in \cite{BB}, the functions $\varphi_\eps$ defined by
$$
\varphi_\eps ' (t) = \exp \left( \int_0^t \exp \Big( - \frac{1}{\eps}  (s- \frac{1}{\eps}) \Big) ds\right).
$$
These functions $\varphi_\eps$ are close to the identity, $\varphi_\eps ' > 0$, $\varphi_\eps '$ converge to $1$ as $\eps \to 0$
and $\varphi_\eps ''$ converge to $0$ as $\eps \to 0$ with $(\varphi_\eps ''(s))^2 > \varphi_\eps ''' (s) \varphi_\eps ' (s)$,
see \cite{BB}.
With $\psi_\eps = \varphi_\eps^{-1}$ we perform the changes of variables
$$
U_{i}^\eps =\psi_\eps (u_i), \qquad V_i^\eps = \psi_\eps (v_i), \qquad i=1,2.
$$
It is clear to see that $U_1$, $V_1$ are semi-convex in space and $U_2$, $V_2$ are semi-concave in space. 
We have that
$$\max \Big\{ \max_{x,t,r} (U_1^\eps (x+h_n,t)- U_2^\eps (x,r) -\frac{|t-r|^2}{2\varepsilon}) ;  \max_{x,t,r} (V_1^\eps (x+h_n,t)- V_2^\eps (x,r)) \Big\}$$
is achieved at some point $x_\eps,t_\eps, r_\eps$. Notice that $U_1$, $V_1$, $-U_2$ and $-V_2$ are bounded from above, 
 $\overline\Omega \times [0,T]$ is compact and $\lim_{t\to T^{-}}U_1\to -\infty$. From standard computations in viscosity theory we get
 \begin{equation} \label{import}
\frac{1}{\varepsilon} |t_{\varepsilon}-r_{\varepsilon}|^{2}\to 0,\quad \mbox{ as } \varepsilon \to 0.
\end{equation}
Extracting a subsequence, if necessary, we can assume that $x_\eps \to x_{h_n}$ and 
$(t_\varepsilon,r_\varepsilon)\to (\widetilde t_n, \widetilde r_n)\in [0,T)^2,$ and we get
$ \widetilde t=\widetilde r$. By using that $u_1(x,t)\leq u_2(x, t)$, $v_1(x,t)\leq v_2(x, t)$ on $\partial\Omega \times (0,T)$ 
and $u_1(x,0) \leq u_2(x,0)$, we have that
$x_h \notin \partial\Omega$ and $\widetilde t=\widetilde r > 0$. Therefore, we get
\begin{equation}\label{interior}
\min\{d(x_h,\partial\Omega)\}>0, \quad \mbox{ and } \quad
\min  \{t_{\varepsilon},\, r_{\varepsilon}\}>0.
\end{equation}

Since we have $|Dw_1(x_n + h_n,\widetilde t_n)| = |Dw_2 (x_n, \widetilde r_n)| > \delta(n)$, 
we deduce that for $\eps$ sufficiently small, it holds that $|DW_1^\eps(x_\eps + h_n, \widetilde t_n)| = |DW_2^\eps (x_\eps, \widetilde r_n)| 
\geq \delta(n)/2$.

 Now, we have, as in \cite{nosotros}, that $U_1$, $V_1$, $U_2$ and $V_2$ verify a strictly monotone system
 (with a strict inequality in the first equation). In fact, the pair
 $(U_1,V_1)$ verifies the equations (in the viscosity sense)
$$
\begin{array}{rl}
\displaystyle 0 >-c & \displaystyle \geq \frac{\partial u_1}{\partial t} - \displaystyle \half \Delta_{\infty}u_1 + u_1 - v_1 \\[10pt]
& \displaystyle = \varphi ' (U_1) \frac{\partial U_1}{\partial t} - \frac12 \varphi ' (U_1)\Delta_\infty U_1 
-  \frac12 \varphi''(U_1)|DU_1|^2  + \varphi(U_1) - \varphi(V_1 )\\[10pt]
& \displaystyle = \varphi ' (U_1) \Big( \frac{\partial U_1}{\partial t} - \frac12 \Delta_\infty U_1 
-  \frac12 \frac{\varphi''(U_1)}{\varphi ' (U_1)}|DU_1|^2  + 
\frac{\varphi(U_1) - \varphi(V_1 )}{\varphi ' (U_1)}\Big),
\end{array}
$$
and 
$$
\begin{array}{rl}
\displaystyle 0 & \displaystyle = -   \displaystyle \frac{\kappa}{2} \Delta v + v - u \\[10pt]
& =  \displaystyle - \frac{\kappa}{2}  \Big( \varphi ' (V_1)\Delta V_1 +  \varphi''(V_1)|DV_1|^2  \Big) + \varphi(V_1 ) - \varphi(U_1  ) \\[10pt]
& =  \displaystyle  \varphi ' (V_1) \Big(- \frac{\kappa}{2} \Delta V_1 - \frac{\kappa}{2}   \frac{\varphi''(V_1)}{\varphi ' (V_1)}
|DV_1|^2 (x) + \frac{\varphi(V_1 ) - \varphi(U_1  )}{ \varphi ' (V_1)}  \Big),
\end{array}
$$
and similar equations also hold for $(U_2,V_2)$.
 Thus, 
from the strict monotonicity, the strict inequality for the parabolic equation and using \eqref{import}, we get the desired contradiction. 
See the proof of \cite{BB}, Lemma 3.1, for a more detailed discussion.
\end{proof}

\medskip

{\bf Acknowledgements.} partially supported by 
CONICET grant PIP GI No 11220150100036CO
(Argentina), PICT-2018-03183 (Argentina) and UBACyT grant 20020160100155BA (Argentina).


\end{document}